\documentclass[11pt]{article}
\usepackage[margin=1in]{geometry}
\usepackage{amsmath,amssymb,amsthm}
\usepackage{booktabs}
\usepackage{graphicx}
\usepackage{tikz}
\usetikzlibrary{decorations.pathreplacing}
\usepackage[dvipsnames]{xcolor}
\usepackage[colorlinks=true,linkcolor=MidnightBlue,citecolor=MidnightBlue,
            urlcolor=MidnightBlue]{hyperref}
\usepackage{url}

\newtheorem{theorem}{Theorem}[section]
\newtheorem{lemma}[theorem]{Lemma}
\newtheorem{proposition}[theorem]{Proposition}
\newtheorem{corollary}[theorem]{Corollary}
\newtheorem{conjecture}[theorem]{Conjecture}
\theoremstyle{definition}
\newtheorem{definition}[theorem]{Definition}
\theoremstyle{remark}
\newtheorem{remark}[theorem]{Remark}

\newcommand{\OFG}{\mathrm{OFG}}
\newcommand{\dgrid}{d_1}
\newcommand{\Z}{\mathbb{Z}}
\newcommand{\Q}{\mathbb{Q}}
\newcommand{\Mm}{M_{m,n}}
\newcommand{\Gmn}{G_{m,n}}
\newcommand{\Aset}{\mathcal{A}}

\title{One construction for the Miura-ori flip-graph degree sequence}
\author{Chakshu Gupta\\
  \small College of Computing, Georgia Institute of Technology\\
  \small \texttt{cgupta65@gatech.edu}}
\date{}

\begin{document}
\maketitle

\begin{abstract}
The flip graph of an origami crease pattern has the locally flat-foldable
mountain-valley assignments as vertices, and an edge joins two of them that differ
by a single face flip. A basic invariant of this graph is the degree sequence,
which counts the vertices of each degree. On the $m\times n$ Miura-ori,
this sequence is known to be a bivariate polynomial only for small degrees,
each count obtained by a separate argument. This paper gives
one uniform construction that expresses, for every degree $d$, the number of
degree-$d$ vertices as a single symmetric polynomial in $(m,n)$ for all
sufficiently large $m,n$. Its degree in each variable is $d-2$ unconditionally.
Subject to a single degree bound, its total degree is $d-2$ as well,
with top-degree part an explicit multiple of
$m^{d-2}+n^{d-2}$ for $d\ge5$. The bound is proved here when the count splits into
independent row and column factors, and open otherwise. The
region is $m,n\ge\max(d-1,2)$. The
polynomials are given in closed form through $d=10$, unconditional through $d=8$,
where the degree bound holds in every case, and conditional on it beyond.
Below this region the count departs from the polynomial. One step below, this
departure has leading coefficient $-4$ times a Baxter number through $d=11$. Each
such polynomial thus counts the Miura-ori's locally flat-foldable assignments
admitting exactly $d$ single face flips.
\end{abstract}

% --- CT submission metadata: map to ct.sty \keywords{} and \MSC{} at conversion ---


% ===========================================================================
% REOPENED 2026-08-04: settled range moved d=7 -> d=8, open range d>=8 -> d>=9.
% Introduction RESTRUCTURED + RE-LOCKED 2026-07-12 (Option B, 4 paragraphs:
% field | context/gap | construction | results). Adversary stress-test PASS:
% 0 CRITICAL, 0 WRONG, 2 SUSPICIOUS (acceptable), 3 MINOR; all 18 atomic claims
% + 9 citations of the prior lock preserved, boundary/Baxter sentence ADDED
% (verbatim-in-scope with the abstract). Findings: /tmp/adversary-intro-optB.md.
% Prior lock 2026-06-28 (2-para) + post-fold refresh 2026-07-05 superseded.
% REWRITTEN 2026-08-24 (U2) to the v2 prose standard: removed the displayed
%   Theorem-A block, the author-named related-work paragraph, and the flourish
%   opener (all v3 CT-fix drift). Main result stated in prose; all v3 substance
%   kept (locally flat-foldable, m,n>=2, lem:degextrema, d=8 via prop:converse,
%   open d>=9, subdir codebase URL). Curious+Adversary converged.
% ===========================================================================
\section{Introduction}\label{sec:intro}
% ===========================================================================

A crease pattern's locally flat-foldable mountain-valley assignments are the
vertices of its origami flip graph $\OFG$~\cite{Hull2022maximal}, with an edge
between two that differ by a single face flip~\cite{Akitaya2020faceflips}. A
vertex's degree is then the number of face flips it admits. For the $m\times n$
Miura-ori $\Mm$~\cite{Miura1994map}, a basic invariant of $\OFG(\Mm)$ is its
degree sequence, the number of vertices of each degree.

The count is known for every $d$ at $m=2$~\cite{Christensen2025origami}, and
for $d\le5$ at all sufficiently large $m,n$, where it is a symmetric
polynomial in $(m,n)$ of total degree $d-2$~\cite{Gupta2026companion}. Each was
proved by a separate argument, the first only at $m=2$ and the second only for
$d\le5$. Whether the count is a polynomial for every $d$, and whether one
construction can prove it uniformly, are the questions addressed here.

Both questions are answered by one construction, built on the height-function
reduction of Section~\ref{sec:background}. Under the Ginepro--Hull
bijection~\cite{GineproHull2014counting} and the bipartite height-function
lift~\cite{Cereceda2009mixing}, $\OFG(\Mm)$ is identified for $m,n\ge2$ with the
integer height functions on the $m\times n$ grid, where a vertex's degree equals
its number of strict local extrema (Lemma~\ref{lem:degextrema}). The envelope
encoding (Theorem~\ref{thm:envelope}) represents each height function by an
integer configuration, so counting the configurations with $d$ extrema is a
parametric lattice-point problem whose count is piecewise quasi-polynomial in $(m,n)$ by the
Barvinok--Woods theory~\cite{BarvinokWoods2003,BogartWoods2022}
(Theorem~\ref{thm:quasipoly}). On the high region $m,n\ge\max(d-1,2)$ this
quasi-polynomial collapses to a single symmetric polynomial $p_d(m,n)$, whose
existence, symmetry, region, and per-variable degree $d-2$ are unconditional
(Theorem~\ref{thm:poly}). Subject to a degree bound (Conjecture~\ref{conj:G2}),
its total degree is $d-2$ as well, with top-degree part
$4/(d-2)!\,(m^{d-2}+n^{d-2})$ for $d\ge5$ (Proposition~\ref{prop:leading}).

These polynomials are computed in closed form through $d=10$
(Section~\ref{sec:data}). The degree bound is proved for the separable
configurations (Theorem~\ref{thm:sepcase}), where the count splits into
independent row and column factors. It holds in every case through $d=8$. Direct
enumeration through $d=7$ establishes it in a self-contained
codebase.\footnote{\url{https://github.com/ChakshuGupta13/lab/tree/main/math/gupta2026degseq}}
At $d=8$, the top-degree part of $p_8$, determined without
assuming the bound, forces it (Proposition~\ref{prop:converse}). The
non-separable case remains open for $d\ge9$. When exactly one of $m,n$ falls
below the high region, the count departs from $p_d$ by a polynomial in the
larger variable, determined by the smaller (Proposition~\ref{prop:meta}). One
step below the region, this polynomial's leading coefficient is, through degree
eleven, $-4$ times a Baxter number (Section~\ref{sec:data}).

% ===========================================================================
\section{The Miura-ori and its flip graph}\label{sec:miura}
% Section 2 ADDED 2026-08-02 in response to the Combinatorial Theory desk
%   rejection (not self-contained; no definition of the main object of study).
%   Every claim grounded in ginepro2014-counting-miura-ori.txt (R7-verified):
%   crease pattern + MV assignment l.80-83; herringbone array + angles
%   alpha,alpha,pi-alpha,pi-alpha l.73-75; Kawasaki + the Miura alternating sum
%   l.100-105; Maekawa l.106-108; bird's-foot forcing and the count six
%   l.110-117; the locally-flat-foldable definition l.124-128; the
%   local-not-global counterexample l.530-541 (GH's own "5x1" figure label is
%   NOT reproduced: under Def 2.1 an m x n Miura-ori has (m-1)(n-1) interior
%   degree-four vertices, which is 0 when n=1, so no bird's-foot vertex exists
%   there and the multi-vertex obstruction cannot arise).
% ===========================================================================

Origami folds a flat sheet of paper along straight lines. Unfolded, the sheet
retains those lines as creases, each recording a convex fold as a mountain or a
concave fold as a valley. A \emph{crease pattern} is a connected
graph $C$ embedded in the plane with straight-line edges, its creases. A
\emph{mountain-valley assignment} on $C$ is a map
$\mu\colon E(C)\to\{\mathrm{M},\mathrm{V}\}$~\cite{GineproHull2014counting}.

\begin{definition}[Miura-ori]\label{def:miura}
Fix an angle $\alpha$ with $0<\alpha<\pi/2$ and integers $m,n\ge2$. The $m\times n$ Miura-ori $\Mm$ is
the crease pattern of $mn$ congruent parallelograms in a herringbone array, with
corners $v_{i,j}$ for $0\le i\le m$ and $0\le j\le n$. Its creases are the fold
lines shared by two of the parallelograms, the zigzag row lines and straight
column lines interior to the sheet. These meet four at a time at each of the $(m-1)(n-1)$
interior vertices, at angles $\alpha,\alpha,\pi-\alpha,\pi-\alpha$ in cyclic
order. The sheet's outer boundary carries no
crease~\cite{GineproHull2014counting}. The hypothesis $m,n\ge2$ excludes strips, which have
no interior vertex.
\end{definition}

A concrete instance appears in Figure~\ref{fig:miura}. Neither the angle $\alpha$ nor
the side ratio of the parallelograms figures in what follows, which uses $\Mm$ only
through its combinatorial structure.

\begin{figure}[ht]
\centering
\begin{tikzpicture}[x=1.2cm, y=1.2cm, font=\small]
  \fill[gray!18] (2.3,0.95) -- (3.45,1.23) -- (3.45,2.18) -- (2.3,1.9) -- cycle;
  \foreach \yb in {0.95, 1.9}
    \draw[thick] (0,\yb) -- (1.15,\yb+0.28) -- (2.3,\yb)
                 -- (3.45,\yb+0.28) -- (4.6,\yb);
  \foreach \yb in {0, 2.85}
    \draw[gray!55] (0,\yb) -- (1.15,\yb+0.28) -- (2.3,\yb)
                 -- (3.45,\yb+0.28) -- (4.6,\yb);
  \draw[thick] (1.15,0.28) -- (1.15,3.13);
  \draw[thick] (2.3,0) -- (2.3,2.85);
  \draw[thick] (3.45,0.28) -- (3.45,3.13);
  \draw[gray!55] (0,0) -- (0,2.85);
  \draw[gray!55] (4.6,0) -- (4.6,2.85);
  \fill (2.3,0.95) circle (1.6pt);
  \node[font=\footnotesize, anchor=east] at (2.20,0.82) {$v_{1,2}$};
  \node[font=\scriptsize] at (2.58,1.31) {$\alpha$};
  \node[font=\scriptsize] at (2.69,0.64) {$\pi-\alpha$};
  \node[font=\footnotesize] at (3.05,1.78) {$(2,3)$};
\end{tikzpicture}
\caption{The Miura-ori $M_{3,4}$: three rows and four columns of congruent
parallelograms, with zigzag row lines and straight column lines. At the interior
vertex $v_{1,2}$ the two acute angles $\alpha$ are adjacent, as are the two
obtuse angles $\pi-\alpha$. One of each is marked. The lighter outer edges bound
the sheet and carry no crease. The shaded face is cell
$(2,3)$ of the grid $\Gmn$ of Section~\ref{sec:background}, with $v_{1,2}$ at one
corner.}
\label{fig:miura}
\end{figure}
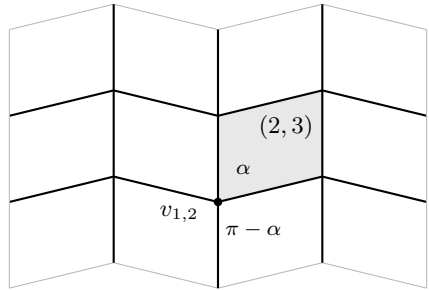

A vertex of a crease pattern folds flat when the paper around it can be pressed
into a book without crumpling and without adding creases. Two classical
conditions govern this. A vertex folds flat iff the alternating sum of its angles, taken in cyclic
order, vanishes~\cite[Thm.~1]{Hull2003counting}. Every vertex of $\Mm$ satisfies this condition, since
$\alpha-\alpha+(\pi-\alpha)-(\pi-\alpha)=0$. At a vertex that folds flat, the numbers
of mountain and valley creases differ by exactly
two~\cite[Thm.~2]{Hull2003counting}. Each interior vertex of $\Mm$ has degree four, so with the counts differing by two it carries three
mountains and one valley, or three valleys and one mountain. Of the eight
assignments this permits at a vertex, the angles rule out two, leaving exactly
six, the bird's-foot forcing~\cite[\S2]{GineproHull2014counting}. An assignment $\mu$ is
\emph{locally flat-foldable} if it restricts at every interior vertex to one of
the six.

\begin{definition}[Face flip and the origami flip graph]\label{def:ofg}
A face flip~\cite{Akitaya2020faceflips} switches every crease bounding a single
face of $C$ between mountain and valley. The origami flip graph $\OFG(C)$ has
the locally flat-foldable assignments of $C$ as its
vertices~\cite{Hull2022maximal}, with an edge joining two that differ by a single face flip.
\end{definition}

The Ginepro--Hull bijection carries the locally flat-foldable
assignments of
$\Mm$ to the proper $3$-colourings of the $m\times n$ grid in which a chosen
corner cell has a fixed colour, sending each face of $\Mm$ to a cell of that grid
and each face flip to a change of colour at the corresponding
cell~\cite{GineproHull2014counting,Christensen2025origami}. Without that
normalisation the three colourings related by a global colour rotation describe
one assignment. Section~\ref{sec:background} develops this into the
height-function model used throughout.

Local flat-foldability constrains each vertex on its own and does not make the
whole sheet fold. Some locally flat-foldable Miura-ori assignments are realised
by no folding of the sheet. The obstruction is layer orders forced
incompatibly across multiple vertices~\cite[\S5]{GineproHull2014counting}. Which assignments fold globally is
a separate question, not addressed here, so the vertices of $\OFG(\Mm)$ are the
locally flat-foldable assignments throughout.

% ===========================================================================
\section{The height-function reduction}\label{sec:background}
% REOPENED 2026-08-04: degree = #extrema was imported from the companion by
%   citation; it is now proved in-paper as lem:degextrema, from the face-flip /
%   colour-change correspondence stated in Section 2. The Ginepro-Hull and
%   Cereceda bijections remain cited, as external results should be.
% Section 3 LOCKED 2026-06-28 (post-Adversary final scrutiny PASS:
%   0 CRITICAL, 0 WRONG; all citations verified against source; Option A
%   Option A applied for the hypothesis on degree=#extrema, CORRECTED
%   2026-08-04 from mn>=3 to m,n>=2: a 1 x n Miura-ori has no interior vertex,
%   so the reduction itself fails there, not merely the injectivity step.
% +Figure 1 (3-cone envelope on G_{5,5}) added 2026-06-30 post-Adversary
%   (APPROVED-WITH-FIXES: dgrid macro + drop |Q|!=|P| overclaim + name
%   meeting cells).
% ===========================================================================

Composing the Ginepro--Hull bijection~\cite{GineproHull2014counting} to the
grid $3$-colourings of Section~\ref{sec:miura} with the bipartite
height-function lift of those colourings~\cite{Cereceda2009mixing} identifies
the vertices of $\OFG(\Mm)$ with integer height functions on the $m\times n$
grid~\cite{Gupta2026companion}. This grid, $\Gmn$, has cells $(i,j)$ with
$1\le i\le m$ and $1\le j\le n$, two of them adjacent when they differ by one in
a single coordinate. Their grid distance is
$\dgrid\big((i,j),(i',j')\big)=|i-i'|+|j-j'|$, the Manhattan distance. A
height function is a map $h\colon\Gmn\to\Z$ with $|h(u)-h(v)|=1$ on each
edge, normalised so that $h(1,1)=0$. A cell is a strict local minimum of
$h$ if every neighbour exceeds it by one and a strict local maximum if
every neighbour falls short by one; the minima and maxima are the extrema
of $h$.

The following lemma recasts the degree sequence as a count of height
functions, and every later section works with these in place of $\OFG(\Mm)$.

\begin{lemma}[Degree and extrema]\label{lem:degextrema}
Let $m,n\ge2$, and let $h$ be the height function of a vertex of $\OFG(\Mm)$. The
neighbours of that vertex correspond to the extrema of $h$, one apiece. The
degree of the vertex is therefore the number of extrema of $h$.
\end{lemma}

\begin{proof}
A face flip changes the colour of exactly one cell
(Section~\ref{sec:miura}), so on height functions it alters $h$ at a single cell
$v$ and leaves the rest fixed. When $v$ is $(1,1)$, a global shift also restores
$h(1,1)=0$.
The result is again a height function exactly when the new value $h'(v)$ differs
by one from $h(u)$ at every neighbour $u$ of $v$, and $h'(v)\ne h(v)$.

\medskip\noindent
Every neighbour $u$ of $v$ has $h(u)=h(v)\pm1$. Suppose two of them disagree,
say $h(u_1)=h(v)-1$ and $h(u_2)=h(v)+1$. The only value at distance one from
both $h(v)-1$ and $h(v)+1$ is $h(v)$ itself, so no flip is available at $v$.
Suppose instead every neighbour carries a common height $t$, which forces
$h(v)=t\pm1$. The values at distance one from $t$ are $t-1$ and $t+1$, so
$h'(v)=2t-h(v)$ is the sole admissible change and the flip at $v$ is unique. When
the neighbours share height $t$, $h(v)=t-1$ makes $v$ a strict local minimum and
$h(v)=t+1$ a strict local maximum; conversely, an extremum has all its
neighbours at a common height. The cells admitting a flip are therefore exactly
the extrema of $h$.

\medskip\noindent
Distinct extrema give distinct neighbours. A flip at a cell other than $(1,1)$
alters $h$ there and nowhere else, so two such flips move different cells. The
flip at $(1,1)$ would change $h(1,1)$ by $\pm2$; renormalising to hold
$h(1,1)=0$ shifts every other cell by $\mp2$ instead. A flip at another extremum
$w$ moves only $w$. Since $m,n\ge2$, the grid has a third cell, distinct from
$(1,1)$ and $w$; the $(1,1)$ flip moves that cell, while the flip at $w$ fixes
it. Each extremum
thus contributes one neighbour, and different extrema contribute different ones.
\end{proof}

To reconstruct $h$ from its extrema and their values, let
$P \subseteq \Gmn$ be the set of cells at which $h$ attains a strict
local minimum. A cone with apex
$p\in\Gmn$ and offset $c\in\Z$ is the function $v\mapsto c+\dgrid(p,v)$,
taking value $c$ at the apex, its lowest point. For $p\in P$, setting the offset to $h(p)$
makes the cone agree with $h$ at $p$, and $h$ is then the pointwise minimum of
these cones (Figure~\ref{fig:kcone-envelope}).

\begin{lemma}[Cone envelope]\label{lem:coneenv}
Let $h$ be a height function on $\Gmn$, let $P$ be its set of strict local
minima and $Q$ its set of strict local maxima. Then $P$ and $Q$ are nonempty
and
\begin{equation}\label{eq:env}
  h(v)=\min_{p\in P}\big(h(p)+\dgrid(p,v)\big)
      =\max_{q\in Q}\big(h(q)-\dgrid(q,v)\big)\qquad\text{for every }v\in\Gmn.
\end{equation}
\end{lemma}

\begin{proof}
Fix a cell $v$. A shortest path to $v$ from any cell $u$ has $\dgrid(u,v)$
edges, and since $h$ changes by exactly $1$ across each, $h(v)-h(u)\le\dgrid(u,v)$. Restricting $u$ to $P$ yields
$h(v)\le\min_{p\in P}\big(h(p)+\dgrid(p,v)\big)$.

\medskip\noindent
For the reverse inequality, step from $v$ to a neighbour of smaller height and
repeat while such a neighbour exists. Each step lowers $h$ by exactly $1$, so no
cell recurs and the walk halts on the finite grid, at a cell $p$ whose
neighbours all exceed it. That cell is a strict local minimum, so $P$ is nonempty
and contains $p$. The walk takes $h(v)-h(p)$ steps and joins $p$ to $v$, so
$\dgrid(p,v)\le h(v)-h(p)$, hence $h(v)\ge h(p)+\dgrid(p,v)$. Minimising over $P$
gives $h(v)\ge\min_{p'\in P}\big(h(p')+\dgrid(p',v)\big)$. With the upper bound $h(v)\le\min_{p\in P}\big(h(p)+\dgrid(p,v)\big)$,
this gives the first equality.

\medskip\noindent
Negating a height function gives a height function, because the unit edge
differences persist and $-h(1,1)=0$. Negation exchanges strict local minima with
strict local maxima. The strict local minima of $-h$ are therefore $Q$, made nonempty by the descent
argument applied to $-h$. The first equality applied to $-h$
reads $-h(v)=\min_{q\in Q}\big(-h(q)+\dgrid(q,v)\big)$. Multiplying by $-1$ turns
the minimum into a maximum and gives
$h(v)=\max_{q\in Q}\big(h(q)-\dgrid(q,v)\big)$, the second equality.
\end{proof}

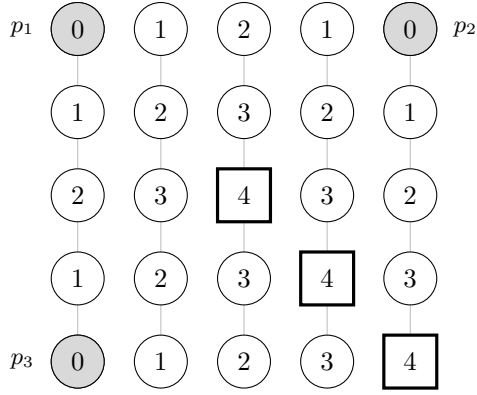
\begin{figure}[ht]
\centering
\begin{tikzpicture}[x=1.1cm, y=-1.1cm, font=\small]
  \draw[gray!50, step=1] (1,1) grid (5,5);
  \foreach \i/\j/\h in {
    1/1/0, 1/2/1, 1/3/2, 1/4/1, 1/5/0,
    2/1/1, 2/2/2, 2/3/3, 2/4/2, 2/5/1,
    3/1/2, 3/2/3, 3/3/4, 3/4/3, 3/5/2,
    4/1/1, 4/2/2, 4/3/3, 4/4/4, 4/5/3,
    5/1/0, 5/2/1, 5/3/2, 5/4/3, 5/5/4
  }
    \node[circle, draw=black, fill=white, minimum size=7mm, inner sep=0pt]
      at (\j,\i) {$\h$};
  \foreach \i/\j/\h in {1/1/0, 1/5/0, 5/1/0}
    \node[circle, draw=black, fill=gray!30, minimum size=7mm, inner sep=0pt]
      at (\j,\i) {$\h$};
  \foreach \i/\j/\h in {3/3/4, 4/4/4, 5/5/4}
    \node[rectangle, draw=black, very thick, fill=white,
          minimum size=7mm, inner sep=0pt] at (\j,\i) {$\h$};
  \node[font=\footnotesize, anchor=east] at (0.6,1) {$p_1$};
  \node[font=\footnotesize, anchor=west] at (5.4,1) {$p_2$};
  \node[font=\footnotesize, anchor=east] at (0.6,5) {$p_3$};
\end{tikzpicture}
\caption{A $3$-cone envelope on $G_{5,5}$ with apexes $p_1=(1,1)$, $p_2=(1,5)$, $p_3=(5,1)$ at offset $0$, so \eqref{eq:env} reads $h(v)=\min_i\dgrid(p_i,v)$. The strict local minima (shaded) are the apexes; the strict local maxima (boxed) sit where cones meet at equal distance. All three cones meet at $(3,3)$, and $p_2$ and $p_3$ alone at $(4,4)$ and $(5,5)$. Section~\ref{sec:maxima} characterises such cells in general.}
\label{fig:kcone-envelope}
\end{figure}

\begin{definition}[Vertex counts]\label{def:counts}
For integers $a,b\ge1$, $N_{(a,b)}(m,n)$ is the number of height functions
on $\Gmn$ with exactly $a$ strict local minima and $b$ strict local maxima,
and $E_d(m,n)$ is the number with exactly $d$ extrema.
\end{definition}

As a count of height functions, $E_d(m,n)$ is defined on the whole region
$mn\ge2$, where every cell has a neighbour and no cell is vacuously both a
minimum and a maximum. On the smaller region $m,n\ge2$, Lemma~\ref{lem:degextrema}
makes it the number of degree-$d$ vertices of $\OFG(\Mm)$, the reading intended
throughout. The high region on which the polynomials live sits inside $m,n\ge2$.

The negation $h\mapsto-h$ exchanges minima and maxima, so
$N_{(a,b)}=N_{(b,a)}$. A height function with $d$ extrema then has $a$ minima and
$b$ maxima with $a+b=d$, and Lemma~\ref{lem:coneenv} makes both sets nonempty, so
$a,b\ge1$. Summing
$N_{(a,b)}$ over $a+b=d$, then folding the pairs $(a,b),(b,a)$ by the
symmetry, gives
\begin{equation}\label{eq:Ed}
  E_d(m,n)=\sum_{\substack{a+b=d\\ a,b\ge1}}N_{(a,b)}(m,n)
          =\sum_{\substack{1\le a\le b\\ a+b=d}}\big(2-\delta_{ab}\big)\,N_{(a,b)}(m,n),
\end{equation}
where $\delta_{ab}$ is the Kronecker delta, $1$ if $a=b$ and $0$ otherwise. This
reduces $E_d$ to the counts $N_{(a,b)}$; the next section encodes each height
function by its minima, recasting these counts as a lattice-point problem.

% ===========================================================================
\section{The Envelope Structure Theorem}\label{sec:envelope}
% Section 4 LOCKED 2026-07-09 (word-level prose scrutiny pass complete: A1
%   activity-condition granularity fix; set B tense + 3 prose colons; set C
%   term consistency, noun/verb, de-personify, opener stall. Builds clean.)
% ===========================================================================

Equation~\eqref{eq:env} expresses every height function as the pointwise minimum of a
family of cones, one for each strict local minimum. This
section makes that minimum an object in its own right, the envelope of a
configuration, and shows that every height function is the envelope of a
unique admissible configuration. Let $m,n$ be integers with $m,n\ge1$ and
$mn\ge2$.

\begin{definition}[Configuration and envelope]\label{def:config}
A configuration on $\Gmn$ is a pair $(A,c)$ where
$A=\{p_1,\dots,p_a\}\subseteq\Gmn$ is a set of $a\ge1$ distinct cells, the
apexes, each written $p_s=(r_s,\kappa_s)$ by row and column; and
$c=(c_1,\dots,c_a)\in\Z^a$ assigns each apex $p_s$ an integer offset $c_s$.
The lower envelope of $(A,c)$ is
\begin{equation}\label{eq:Eac}
  E_{A,c}(v)=\min_{1\le s\le a}\big(c_s+\dgrid(p_s,v)\big).
\end{equation}
\end{definition}

The envelope $E_{A,c}$ is an additively weighted $L^1$ Voronoi diagram of the
apexes~\cite{Aurenhammer1991}, a polyhedral-norm Voronoi diagram of the kind
studied tropically~\cite{CriadoJoswigSantos2021,DevelinSturmfels2004}.

\begin{definition}[Admissible configuration]\label{def:admissible}
Apex $p_s$ is \emph{active} if cone $s$ lies strictly below every other
cone at $p_s$. Cone $s$ has value $c_s$ at its apex $p_s$ and cone $t$ has
value $c_t+\dgrid(p_t,p_s)$, so
\begin{equation}\label{eq:active}
  c_s<c_t+\dgrid(p_t,p_s)\qquad\text{for all }t\ne s.
\end{equation}
The configuration satisfies the \emph{activity condition} if every apex is
active. It satisfies the \emph{parity condition} if
\begin{equation}\label{eq:par}
  c_s-c_t\equiv\dgrid(p_s,p_t)\pmod 2\qquad\text{for all }s,t.
\end{equation}
Equivalently, any two cones share a common parity at every cell, so the envelope
differs across each edge rather than staying level.
The configuration $(A,c)$ is \emph{admissible} if it satisfies the activity
condition, the parity condition, and $E_{A,c}(1,1)=0$.
\end{definition}

Under these conditions, Lemmas~\ref{lem:minatapex}--\ref{lem:activemin}
show that $E_{A,c}$ is a height function whose strict local minima are
exactly the active apexes. Theorem~\ref{thm:envelope} assembles this into a
bijection between height functions and admissible configurations, so
counting the height functions with a prescribed number of minima becomes
counting admissible configurations. The cone-pair bijection
of~\cite{Gupta2026companion} is the two-apex case, matching each
degree-$4$ vertex to an envelope of two cones; Theorem~\ref{thm:envelope}
lifts it to any number of apexes, as counting every degree requires.

\begin{lemma}[Minima at apexes]\label{lem:minatapex}
For a configuration $(A,c)$, every strict local minimum of $E_{A,c}$ is
an apex.
\end{lemma}

\begin{proof}
Let $v$ be a strict local minimum of $E_{A,c}$, and let cone $s^\ast$
attain the minimum $E_{A,c}(v)$, so
$E_{A,c}(v)=c_{s^\ast}+\dgrid(p_{s^\ast},v)$. Suppose, for contradiction,
that $v\ne p_{s^\ast}$. Moving one step from $v$ toward $p_{s^\ast}$
reaches a neighbour $w$ with $\dgrid(p_{s^\ast},w)=\dgrid(p_{s^\ast},v)-1$.
Since $E_{A,c}(w)$ is the minimum of the cone values,
\[
  E_{A,c}(w) \ \le\ c_{s^\ast}+\dgrid(p_{s^\ast},w) \ =\ E_{A,c}(v)-1 \ <\ E_{A,c}(v) .
\]
This contradicts that $v$ is a strict local minimum. Hence
$v=p_{s^\ast}$, an apex.
\end{proof}

\begin{lemma}[Cone parity and unit edge differences]\label{lem:parityvalid}
For a configuration $(A,c)$ satisfying the parity
condition~\eqref{eq:par}, $c_s-c_t\equiv\dgrid(p_s,p_t)\pmod 2$ for all
$s$ and $t$, the cones $c_s+\dgrid(p_s,\cdot)$ share a common parity at
each cell, and the lower envelope $E_{A,c}\colon\Gmn\to\Z$ satisfies
$|E_{A,c}(u)-E_{A,c}(v)|=1$ on every edge $uv$ of $\Gmn$.
\end{lemma}

\begin{proof}
Each cone $c_s+\dgrid(p_s,\cdot)$ is integer-valued, being the sum of an
integer and a non-negative integer, and $1$-Lipschitz by the triangle
inequality on $\dgrid$. A pointwise minimum of integer-valued
$1$-Lipschitz functions is again integer-valued and $1$-Lipschitz, so
$E_{A,c}$ is both. On every edge $uv$, where
$\dgrid(u,v)=1$, the integer $|E_{A,c}(u)-E_{A,c}(v)|$ is at most $1$,
hence $0$ or $1$. Ruling out $0$ amounts to showing that $E_{A,c}$ has
opposite parities at the two endpoints of every edge, as values of
opposite parity differ by an odd amount. Modulo $2$ every integer is
congruent to its negative, so an absolute value may be dropped and a
minus sign flipped, giving $|a-b|\equiv a-b\equiv a+b\pmod 2$ for all integers
$a$ and $b$. With $p_s=(r_s,\kappa_s)$, the value of cone
$s$ at a cell $(i,j)$ is
\[
  c_s+\dgrid(p_s,(i,j))=c_s+|r_s-i|+|\kappa_s-j|
    \equiv(i+j)+(c_s+r_s+\kappa_s)\pmod 2 .
\]
The same identity gives
$\dgrid(p_s,p_t)\equiv(r_s+\kappa_s)+(r_t+\kappa_t)\pmod 2$. Substituting into
the parity condition~\eqref{eq:par} turns $c_s-c_t\equiv\dgrid(p_s,p_t)$ into
$c_s-c_t\equiv(r_s+\kappa_s)+(r_t+\kappa_t)$. As signs are immaterial
modulo $2$, collecting the $s$- and $t$-indexed terms on opposite sides
gives
\[
  c_s+r_s+\kappa_s\equiv c_t+r_t+\kappa_t\pmod 2 .
\]
The summand $c_s+r_s+\kappa_s$ thus has one fixed parity, so at each cell
every cone value has parity $i+j$ plus that fixed amount, and so does
their minimum $E_{A,c}$. Across an edge $i+j$ changes by $1$, flipping
this parity, so $E_{A,c}$ has opposite parities at the endpoints. The
difference $|E_{A,c}(u)-E_{A,c}(v)|$ is therefore odd, hence equal to $1$
rather than $0$.
\end{proof}

\begin{lemma}[Minima Criterion]\label{lem:activemin}
For a configuration $(A,c)$ with apexes $p_1,\dots,p_a$ and offsets
$c_1,\dots,c_a$ satisfying the parity condition~\eqref{eq:par},
$c_s-c_t\equiv\dgrid(p_s,p_t)\pmod 2$ for all $s$ and $t$, an apex $p_s$
is a strict local minimum of $E_{A,c}$ iff it is
active~\eqref{eq:active}, $c_s<c_t+\dgrid(p_t,p_s)$ for every $t\ne s$.
\end{lemma}

\begin{proof}
Suppose first that $p_s$ is active, so that
$c_s < c_t + \dgrid(p_t,p_s)$ for every $t \ne s$; as both sides are
integers, $c_t + \dgrid(p_t,p_s) \ge
c_s + 1$. At $p_s$, cone $s$ has value $c_s + \dgrid(p_s,p_s) = c_s$ and
every other cone has value at least $c_s + 1$, so $E_{A,c}(p_s)$, the
minimum of these cone values, equals $c_s$.
Let $w$ be a neighbour of $p_s$ in $\Gmn$, so $\dgrid(p_s,w) = 1$. At
$w$, cone $s$ has value $c_s + \dgrid(p_s,w) = c_s + 1$, while for each
$t \ne s$ the triangle inequality gives
\[
  c_t + \dgrid(p_t,w) \ \ge\ c_t + \dgrid(p_t,p_s) - 1 \ \ge\ c_s .
\]
Thus $E_{A,c}(w)$, the minimum of these cone values, is at least
$c_s = E_{A,c}(p_s)$. Since $E_{A,c}$ differs by exactly $1$ across each
edge by Lemma~\ref{lem:parityvalid}, this forces
$E_{A,c}(w) = E_{A,c}(p_s) + 1$. Every neighbour of $p_s$ exceeds it, so
$p_s$ is a strict local minimum, establishing that active apexes are
minima.

\medskip\noindent
Conversely, suppose that $p_s$ is not active, so that
$c_t + \dgrid(p_t,p_s) \le c_s$ for some $t \ne s$. If cone $s$ attains
the minimum $E_{A,c}(p_s)$, then $E_{A,c}(p_s) = c_s$. Since
$E_{A,c}(p_s)$ is the minimum of the cone values,
$E_{A,c}(p_s) \le c_t + \dgrid(p_t,p_s)$. The assumption
$c_t + \dgrid(p_t,p_s) \le c_s = E_{A,c}(p_s)$ then forces
$c_t + \dgrid(p_t,p_s) = E_{A,c}(p_s)$, so cone $t$ attains it as well.
Otherwise, some cone other than $s$ attains it. Either way,
$E_{A,c}(p_s) = c_r + \dgrid(p_r,p_s)$ for some $r \ne s$, and since the
apexes are distinct, $p_r \ne p_s$.

\medskip\noindent
Moving one step from $p_s$ toward $p_r$ reaches a neighbour $w$ with
$\dgrid(p_r,w) = \dgrid(p_r,p_s) - 1$. Since $E_{A,c}(w)$ is the minimum
of the cone values,
\[
  E_{A,c}(w) \ \le\ c_r + \dgrid(p_r,w) \ =\ E_{A,c}(p_s) - 1 \ <\ E_{A,c}(p_s) .
\]
A neighbour of $p_s$ therefore has smaller value, so $p_s$ is not a strict
local minimum, establishing by contraposition that minima are active apexes.
\end{proof}

\begin{theorem}[Envelope Structure Theorem]\label{thm:envelope}
Every height function $h$ on $\Gmn$ is the lower envelope $E_{A,c}$ of a
unique admissible configuration $(A,c)$, whose apex set $A$ is the set of
strict local minima of $h$ and whose offsets are $c_s=h(p_s)$. This gives a bijection
between the height functions on $\Gmn$ and the admissible configurations.
\end{theorem}

\begin{proof}
Given a height function $h$ with strict local minima $P$,
equation~\eqref{eq:env} expresses $h$ as the lower envelope of the cones at
$P$. This is $E_{A,c}$ for the configuration with $A=P$ and
$c_s=h(p_s)$, so
$h=E_{A,c}$, and $E_{A,c}(1,1)=h(1,1)=0$ since a height function vanishes
at $(1,1)$. The parity
condition~\eqref{eq:par} holds because a shortest path from $p_t$ to $p_s$ has
$\dgrid(p_s,p_t)$ edges, along each of which $h$ changes by $\pm1$, so
$h(p_s)-h(p_t)\equiv\dgrid(p_s,p_t)\pmod 2$. Each $p_s$, a strict
local minimum, is active by Lemma~\ref{lem:activemin}. Thus $(A,c)$ is
admissible.

\medskip\noindent
Conversely, let $(A,c)$ be admissible. Lemma~\ref{lem:parityvalid} gives
$|E_{A,c}(u)-E_{A,c}(v)|=1$ on every edge of $\Gmn$; these unit edge
differences and $E_{A,c}(1,1)=0$ are the conditions defining a height
function, so $E_{A,c}$ is one.
Lemma~\ref{lem:minatapex} makes every strict local minimum of $E_{A,c}$
an apex. Since every apex is active, Lemma~\ref{lem:activemin}
makes every apex a strict local minimum. The strict local minima are
therefore exactly the apex set $A$. At each apex $p_s$, cone $s$ has value
$c_s$, which lies strictly below every other cone there because $p_s$ is
active, so $E_{A,c}(p_s)=c_s$. The map taking a height function to its strict local
minima and the values it takes there thus sends $E_{A,c}$ to $(A,c)$,
undoing the envelope $(A,c)\mapsto E_{A,c}$. The forward direction's
$E_{A,c}=h$ shows the envelope undoes this map. Each is thus the other's
inverse, so the correspondence is a bijection.
\end{proof}

\begin{corollary}[Parameter count]\label{cor:Odparams}
A height function $h$ on $\Gmn$ with $d$ extrema is determined by at most
$3\lfloor d/2\rfloor-1$ integer parameters, an $O(d)$ count depending only
on $d$, on neither $m$ nor $n$.
\end{corollary}

\begin{proof}
Let $a$ and $b$ be the numbers of strict local minima and maxima of $h$,
so $a+b=d$. By Theorem~\ref{thm:envelope}, $h=E_{A,c}$ for the
configuration whose apexes are the $a$ minima, with offsets $c_s=h(p_s)$.
This configuration is specified by $3a$ integers, $2a$ apex coordinates
and $a$ offsets. Adding a constant to every offset shifts $h$ by that
constant, so the $a$ offsets carry one redundant degree of freedom; the
normalisation $h(1,1)=0$ removes it, leaving $3a-1$ integers to determine
$h$. The negation $h\mapsto-h$ preserves both defining
conditions and exchanges minima with maxima, so $h$ is equally determined
by its $b$ maxima, hence by $3b-1$ integers. The smaller count
$3\min\{a,b\}-1$ therefore applies, and $\min\{a,b\}\le\lfloor d/2\rfloor$
gives the bound $3\lfloor d/2\rfloor-1$.
\end{proof}

% ===========================================================================
\section{The Maxima Criterion}\label{sec:maxima}
% Section 5 LOCKED 2026-07-09 (word-level prose scrutiny pass complete: P1-a
%   Aset(v) attachment; P3-a parity referent; P4-a on-neither-of-the-two unify;
%   P6-a split->coincidence. P2 display-colon + caption data-parens kept.)
% ===========================================================================

Theorem~\ref{thm:envelope} encodes a height function by its minima, the apexes
of its configuration; its maxima are left implicit, determined by the
configuration rather than chosen. Evaluating $N_{(a,b)}$, the number of height
functions with $a$ minima and $b$ maxima, therefore requires a test for whether a
given cell is a maximum. At a
cell $v$, let $\Aset(v)=\{s:c_s+\dgrid(p_s,v)=E_{A,c}(v)\}$ index the cones
attaining the envelope there. The criterion below decides from $\Aset(v)$
whether $v$ is a maximum.

\begin{lemma}[Maxima Criterion]\label{lem:maxima}
Let a configuration $(A,c)$ satisfy the parity condition~\eqref{eq:par},
$c_s-c_t\equiv\dgrid(p_s,p_t)\pmod 2$ for all $s$ and $t$, with apexes
$p_s=(r_s,\kappa_s)$ and envelope $E=E_{A,c}$. A cell $v=(i,j)$ is a strict
local maximum of $E$ iff, in each direction where $v$ has a neighbour, some
attaining cone has its apex in that direction:
\[
\begin{aligned}
\text{down }(i<m):&\ \exists s\in\Aset(v),\ r_s>i; &
\text{up }(i>1):&\ \exists s\in\Aset(v),\ r_s<i;\\
\text{right }(j<n):&\ \exists s\in\Aset(v),\ \kappa_s>j; &
\text{left }(j>1):&\ \exists s\in\Aset(v),\ \kappa_s<j.
\end{aligned}
\]
\end{lemma}

\begin{proof}
By Lemma~\ref{lem:parityvalid}, $E$ has unit edge differences, so every
neighbour $w$ of $v$ has $E(w)=E(v)\pm1$. Hence $v$ is a strict local maximum iff
every neighbour has $E(w)=E(v)-1$. On the step $v\to w$, each cone changes by
$\pm1$. An attaining cone, of value $E(v)$ at $v$, drops to $E(v)-1$ when the
step moves toward its apex and rises to $E(v)+1$ otherwise. A non-attaining cone
is strictly above $E(v)$ at $v$ and shares the parity of $E(v)$ by
Lemma~\ref{lem:parityvalid}, so it is at least $E(v)+2$ there and at least
$E(v)+1$ at $w$. Their minimum $E(w)$ is therefore $E(v)-1$ exactly when some
attaining cone drops, and $E(v)+1$ otherwise. A step toward $p_s=(r_s,\kappa_s)$,
one that decreases $\dgrid(p_s,\cdot)$, is downward when $r_s>i$, upward when
$r_s<i$, rightward when $\kappa_s>j$, and leftward when $\kappa_s<j$. Imposing
each only when $v$ has that neighbour gives the statement's four conditions.
\end{proof}

\begin{corollary}[Single-cone maxima]\label{cor:corners}
For $m,n\ge2$, a cell $v$ with a single attaining cone, $|\Aset(v)|=1$, is a
strict local maximum iff $v$ is a grid corner and the apex of that cone lies on
neither of the two boundary lines through $v$.
\end{corollary}

\begin{proof}
Let $s$ be the single attaining cone at $v$, with apex $p_s=(r_s,\kappa_s)$. If
$v$ is a strict local maximum, then by Lemma~\ref{lem:maxima} the apex points
toward every neighbour of $v$. Both an upper and a lower neighbour would force
$r_s<i$ and $r_s>i$ at once, so $v$ has at most one vertical neighbour and lies
in row $1$ or row $m$. Both a left and a right neighbour would force
$\kappa_s<j$ and $\kappa_s>j$ at once, so $v$ has at most one horizontal
neighbour and lies in column $1$ or column $n$. A single-cone maximum is
therefore a grid corner. At a corner, $v$'s row and column are its two boundary
lines, with one neighbour just inside each. Because $v$'s row bounds the grid, the
apex points toward the vertical neighbour exactly when the apex lies off that row.
Because its column does too, the apex points toward the horizontal neighbour
exactly when the apex lies off that column. By
Lemma~\ref{lem:maxima}, $v$ is a maximum exactly when the apex points toward
both neighbours, hence exactly when the apex lies on neither of the two
boundary lines through $v$.
\end{proof}

\begin{figure}[ht]
\centering
\begin{tikzpicture}[x=1.1cm, y=-1.1cm, font=\small]
  \draw[gray!50, step=1] (1,1) grid (5,5);
  \foreach \i/\j/\h in {
    1/1/0, 1/2/1, 1/3/2, 1/4/1, 1/5/0,
    2/1/1, 2/2/2, 2/3/3, 2/4/2, 2/5/1,
    3/1/2, 3/2/3, 3/3/4, 3/4/3, 3/5/2,
    4/1/1, 4/2/2, 4/3/3, 4/4/4, 4/5/3,
    5/1/0, 5/2/1, 5/3/2, 5/4/3, 5/5/4
  }
    \node[circle, draw=black, fill=white, minimum size=7mm, inner sep=0pt]
      at (\j,\i) {$\h$};
  \foreach \i/\j/\h in {1/1/0, 1/5/0, 5/1/0}
    \node[circle, draw=black, fill=gray!30, minimum size=7mm, inner sep=0pt]
      at (\j,\i) {$\h$};
  \foreach \i/\j/\h in {2/3/3, 4/3/3, 3/2/3, 3/4/3}
    \node[circle, draw=black, fill=gray!15, minimum size=7mm, inner sep=0pt]
      at (\j,\i) {$\h$};
  \node[rectangle, draw=black, very thick, fill=white,
        minimum size=7mm, inner sep=0pt] at (3,3) {$4$};
  \foreach \i/\j/\h in {4/4/4, 5/5/4}
    \node[rectangle, draw=black, very thick, fill=white,
          minimum size=7mm, inner sep=0pt] at (\j,\i) {$\h$};
  \node[font=\footnotesize, anchor=east] at (0.6,1) {$p_1$};
  \node[font=\footnotesize, anchor=west] at (5.4,1) {$p_2$};
  \node[font=\footnotesize, anchor=east] at (0.6,5) {$p_3$};
  \node[font=\footnotesize, anchor=west] at (5.4,3) {$v=(3,3)$};
  \node[font=\scriptsize, anchor=north] at (3, 2.32) {$p_1,p_2$};
  \node[font=\scriptsize, anchor=north] at (3, 4.32) {$p_3$};
  \node[font=\scriptsize, anchor=north] at (2, 3.32) {$p_1,p_3$};
  \node[font=\scriptsize, anchor=north] at (4, 3.32) {$p_2$};
\end{tikzpicture}
\caption{Lemma~\ref{lem:maxima} at $v=(3,3)$ from Figure~\ref{fig:kcone-envelope}. The tag under each of $v$'s four neighbours names the attaining cones pointing that way. The apexes lie above $v$ ($p_1,p_2$, row $1<3$), below ($p_3$, row $5>3$), to the right ($p_2$, column $5>3$), and to the left ($p_1,p_3$, column $1<3$), so every direction is covered and $v$ is a strict local maximum. Cells $(4,4)$ and $(5,5)$ are maxima too, each attained by $p_2$ and $p_3$ alone; being a corner, $(5,5)$ has only two neighbours to cover.}
\label{fig:max-criterion}
\end{figure}
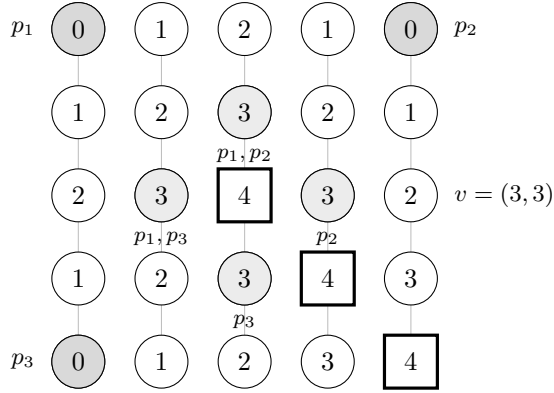

By Corollary~\ref{cor:corners}, for $m,n\ge2$ every single-cone
maximum is a grid corner, hence at most four. Every other maximum has
$|\Aset(v)|\ge2$, so lies where two of its attaining cones coincide,
$c_s+\dgrid(p_s,\cdot)=c_t+\dgrid(p_t,\cdot)$ (Figure~\ref{fig:max-criterion}).
The same coincidence appears in the Ridge Lemma of~\cite{Gupta2026companion}, whose
two apexes must differ in both coordinates; Lemma~\ref{lem:maxima} needs no such
restriction and holds for any number of apexes.

% ===========================================================================
\section{Joint quasi-polynomiality}\label{sec:quasipoly}
% Section 6 LOCKED 2026-07-09 (word-level prose scrutiny pass complete: P1-a
%   pronoun antecedent; P5-a + P9 prose colons; P8-a hedge word; P9 citation
%   genitive. Builds clean.)
% ===========================================================================

The envelope encoding (Theorem~\ref{thm:envelope}) makes each height function a
configuration, a tuple of integers recording its apex coordinates and offsets,
with its minima as the apexes. The Maxima Criterion (Lemma~\ref{lem:maxima})
determines its maxima. For each $(m,n)$, a height function on the finite grid
takes bounded values, so there are finitely many configurations with prescribed
numbers of minima and maxima. For $d\le5$, $E_d$ is known explicitly, each by a
separate argument~\cite{Gupta2026companion}. Counting the configurations as $(m,n)$ varies makes
$E_d$ piecewise quasi-polynomial for every $d$, by a single
argument~\cite{BogartWoods2022}. This is weaker than the
explicit polynomials of the small cases but uniform in $d$, and for each $d$ the
later sections sharpen this piecewise quasi-polynomial into a single polynomial
on a high region.

This count of configurations rests on a counting theorem for Presburger families, which the following definitions set up.
A Presburger family is a collection of sets $S_t$, one for each integer
parameter vector
$t\in\Z^k$, where $S_t$ is the set of points $x\in\Z^N$ satisfying a formula built
from linear inequalities $\alpha\cdot x\le\beta+\gamma\cdot t$ with integer
coefficients, using Boolean connectives and quantifiers over auxiliary
integer variables~\cite[Def.~7]{BogartWoods2022}.
A function $g:\Z^k\to\Z$ is a quasi-polynomial if it agrees with a
polynomial on each residue class modulo some period, and
piecewise quasi-polynomial if $\mathbb{R}^k$ can be split into finitely many
regions cut out by linear inequalities, with $g$ agreeing with a
quasi-polynomial on each region~\cite[Defs.~4--5]{BogartWoods2022}. If $\{S_t\}$ is a
Presburger family, its counting function $g(t)=\lvert S_t\rvert$ is piecewise
quasi-polynomial wherever it is finite~\cite[Thm.~1.10]{Woods2015}; see
also~\cite[Thm.~2, Property~2]{BogartWoods2022}.

\begin{theorem}[Joint quasi-polynomiality]\label{thm:quasipoly}
For each fixed $d\ge2$, $E_d(m,n)$ is piecewise quasi-polynomial in $(m,n)$
on $\{m,n\ge1\}$.
\end{theorem}

\begin{proof}
The count $E_d$ splits over the finitely many $(a,b)$ with $a+b=d$. For each split, admissible configurations with $a$ minima and $b$ maxima form a Presburger family whose counting function is $N_{(a,b)}$. The counting theorem for Presburger families makes $N_{(a,b)}$ piecewise quasi-polynomial, and $E_d$ is their finite sum wherever $mn\ge2$. The single grid $G_{1,1}$ is treated on its own.

\medskip\noindent\emph{Configuration space.}
A split $(a,b)$ divides the $d$ extrema into $a$ minima and $b$ maxima.
By Theorem~\ref{thm:envelope} a height
function is a configuration whose apexes are its minima, so $N_{(a,b)}$ counts
configurations with $a$ apexes. Such a
configuration is a point $(A,c)\in\Z^{3a}$, two coordinates for each apex cell
$(r_s,\kappa_s)\in[1,m]\times[1,n]$ and one offset $c_s$ per apex, the apexes in a fixed order. The
normalisation $E_{A,c}(1,1)=0$, imposed as a constraint below, removes one of
these degrees of freedom (Corollary~\ref{cor:Odparams}).

\medskip\noindent\emph{Presburger encoding.}
The point $(A,c)$ ranges over $\Z^{3a}$, and each defining condition is a Presburger formula in $(A,c)$ and the parameters $(m,n)$.
Each apex cell obeys the grid bounds $1\le r_s\le m$ and $1\le\kappa_s\le n$, and the apexes are listed in lexicographic order, $(r_s,\kappa_s)\prec(r_{s+1},\kappa_{s+1})$ for $s<a$, so each configuration is a single point.
There is one active inequality~\eqref{eq:active} for each apex $s$ against each of the other $a-1$ apexes, giving $a(a-1)$ in all. Each involves a grid distance $\dgrid$, a sum of absolute values, hence expands to a Boolean combination of linear inequalities.
The $\binom{a}{2}$ parity congruences~\eqref{eq:par} are symmetric under $s\leftrightarrow t$, as $\dgrid$ is symmetric and $c_s-c_t\equiv c_t-c_s\pmod 2$.
The normalisation $E_{A,c}(1,1)=0$ is linear at the corner, where $\dgrid(p_s,(1,1))=r_s+\kappa_s-2$. It reads $\bigwedge_s(c_s+r_s+\kappa_s\ge2)\wedge\bigvee_s(c_s+r_s+\kappa_s=2)$.
The split further requires exactly $b$ maxima. By
Lemma~\ref{lem:maxima}, being a maximum is a criterion $\Phi(A,c,v)$, a Boolean combination of linear inequalities in $(A,c,v)$. The requirement reads
\[
  \exists\,v_1,\dots,v_b\ \text{distinct}\ \textstyle\bigwedge_i\Phi(\cdot,v_i)
  \ \wedge\ \neg\,\exists\,v_1,\dots,v_{b+1}\ \text{distinct}\
  \textstyle\bigwedge_i\Phi(\cdot,v_i),
\]
where the $v_i$ range over the grid and the quantifiers are bounded by the parameters $(m,n)$. With parameters
$t=(m,n)$ and points $x=(A,c)$, these conditions combine into one such formula, defining $S_{(m,n)}$ as the set of admissible $(A,c)$ with exactly $b$ maxima. The family
$\{S_{(m,n)}\}$ is thus a Presburger family.

\medskip\noindent\emph{Finiteness.}
By the active inequalities $c_s<c_t+\dgrid(p_t,p_s)$ and $c_t<c_s+\dgrid(p_s,p_t)$, each $\lvert c_s-c_t\rvert$ is at most $\dgrid(p_s,p_t)$. Since $p_s,p_t\in[1,m]\times[1,n]$, this is at most $(m-1)+(n-1)$. The offsets are then determined up to a common shift, which the normalisation $E_{A,c}(1,1)=0$ pins. Over the finitely many apex placements on $\Gmn$, the bounded offsets make $S_{(m,n)}$ finite, and $\lvert S_{(m,n)}\rvert=N_{(a,b)}(m,n)$.

\medskip\noindent\emph{Assembly.}
As the counting function of this Presburger family, $N_{(a,b)}$
is piecewise quasi-polynomial by~\cite[Thm.~1.10]{Woods2015}. A finite sum of
piecewise quasi-polynomials is itself piecewise quasi-polynomial, on a common refinement of the partitions and with a
common period. By~\eqref{eq:Ed}, which sums the $N_{(a,b)}$ over the splits of $d$,
$E_d$ is piecewise quasi-polynomial on $\{mn\ge2\}$. Within $\{m,n\ge1\}$, the set $\{mn\ge2\}$
is the union of the two linear regions $\{m\ge2,\,n\ge1\}$ and
$\{m=1,\,n\ge2\}$. The one remaining grid,
$G_{1,1}$, is the single lattice point cut out by $m\le1$ and $n\le1$, so $E_d$ is
constant there. It is excluded above only because its lone cell has no
neighbour and is therefore vacuously both a minimum and a maximum.
Adjoining it as its own piece extends the conclusion to $\{m,n\ge1\}$.
\end{proof}

Both~\cite{Gupta2026companion} and this paper extend the $m=2$ count of~\cite{Christensen2025origami} to general $(m,n)$. The former uses a separate explicit polynomial for each small $d$, the latter a single argument giving piecewise quasi-polynomiality for every $d$.
What remains is to upgrade this piecewise quasi-polynomial to a single
polynomial on $\{m,n\ge\max(d-1,2)\}$. Sections~\ref{sec:onedim}--\ref{sec:leading} determine its degree and top-degree part, Section~\ref{sec:period} its period, and
Section~\ref{sec:data} the threshold and the boundary corrections below it.

% ===========================================================================
\section{Single-side configurations}\label{sec:onedim}
% Section 7 LOCKED 2026-07-10 (word-level prose scrutiny re-walk complete: P1
%   opener adjective drop; P2 lem:edgered statement colon; P3 maxima-claim
%   grounding. P4-P9 bridge, lem:binom, figure caption reviewed clean. Builds clean.)
% ===========================================================================

Upgrading the piecewise quasi-polynomial of Section~\ref{sec:quasipoly} to a
polynomial begins at the boundary of the grid. A configuration is
\emph{single-side} if all $a$ apexes lie on a single side, the top or bottom
row, or the left or right column. Such a configuration collapses to a one-dimensional
$\pm1$ walk (Lemma~\ref{lem:edgered}). Counting the single-side configurations is then
counting walks, which Lemma~\ref{lem:binom} enumerates in closed form. Whether this single-side
count is the entire top-degree part of $E_d$ is the question of Section~\ref{sec:leading}.

\begin{lemma}[Side reduction]\label{lem:edgered}
If a configuration $(A,c)$ satisfies the activity condition~\eqref{eq:active},
$c_s<c_t+\dgrid(p_t,p_s)$ for all $s$ and every $t\ne s$, and the parity
condition~\eqref{eq:par}, $c_s-c_t\equiv\dgrid(p_s,p_t)\pmod 2$ for all $s$ and
$t$, with all $a$ apexes on the top row $i=1$ at columns
$\kappa_1<\dots<\kappa_a$, then
$E_{A,c}(i,j)=(i-1)+\tau(j)$, where $\tau(j)=\min_s\big(c_s+|j-\kappa_s|\big)$
is the one-dimensional lower envelope on the path $P_n=\{1,\dots,n\}$. Consequently,
$E_{A,c}$ increases strictly down each column. Its minima are the row-$1$ apexes,
and its maxima are the bottom-row cells at the strict local maxima of $\tau$, so
$E_{A,c}$ has as many maxima as $\tau$. The conditions on such configurations do
not involve $m$, so their number is the same for every $m\ge2$.
\end{lemma}

\begin{proof}
Because every apex lies in row $1$, the distance from $(1,\kappa_s)$ to $(i,j)$
splits into the vertical drop $i-1$ and the horizontal distance $|j-\kappa_s|$. Hence
\[
  E_{A,c}(i,j)=\min_s\big(c_s+(i-1)+|j-\kappa_s|\big)=(i-1)+\tau(j),
\]
because the $i-1$ term is constant in $s$ and factors out of the minimum. By Lemma~\ref{lem:parityvalid} the envelope is a
height function, differing by exactly $1$ across each edge. Across the horizontal
edge from $(i,j)$ to $(i,j+1)$ this gives $|\tau(j+1)-\tau(j)|=1$, so $\tau$ steps
by $\pm1$. The summand $i-1$ makes $E_{A,c}$ increase strictly down each column, so its
minima lie in row $1$ and, by Lemmas~\ref{lem:minatapex} and~\ref{lem:activemin},
are exactly the apexes. Its maxima lie in row $m$, at the strict local maxima
of $\tau$, which the constant offset $m-1$ does not move. On the top row,
the activity and parity conditions constrain only the columns $\kappa_s$ and
offsets $c_s$, not $m$, so the valid configurations, and hence their number, are
the same for every $m\ge2$. The
other three sides follow by symmetry. Reflecting the grid reaches the bottom row,
and transposing it reaches the side columns.
\end{proof}

By Lemma~\ref{lem:edgered}, each single-side configuration reduces to a $\pm1$
walk, up to an additive constant. The following lemma counts such
walks, the maps $h\colon P_n\to\Z$ with $|h(j+1)-h(j)|=1$, by their numbers of
minima and maxima.

\begin{lemma}[Walk count]\label{lem:binom}
For integers $a,b\ge1$, let $f_{(a,b)}(n)$ be the number of $\pm1$ walks on $P_n$
with exactly $a$ strict local minima and $b$ strict local maxima, where the
endpoints are counted as extrema. Then $f_{(a,b)}(n)=0$ unless $|a-b|\le1$, and
for $|a-b|\le1$,
\begin{equation}\label{eq:fclosed}
  f_{(a,b)}(n)=\big(1+\delta_{ab}\big)\binom{n-2}{d-2},\qquad d=a+b.
\end{equation}
\end{lemma}

\begin{proof}
Encode $h$ by its step string $s\in\{\pm1\}^{n-1}$, whose entries
$s_j=h(j+1)-h(j)$ determine $h$ up to an additive constant. The
maximal monotone runs of $s$ alternate up and down. With $r$ runs, the $r-1$
interior turning points and the two endpoints are all extrema, so $a+b=(r-1)+2=r+1$
and $r=d-1$. Along the walk, the $d$ extrema alternate minimum and maximum, so
$a$ and $b$ differ by at most one and $f_{(a,b)}(n)=0$ when $|a-b|>1$. The
alternating pattern is fixed by the direction of the first run
(Figure~\ref{fig:walkcount}). Its two choices realise the splits $(a,b)$ and
$(b,a)$, which coincide iff $a=b$, so $1+\delta_{ab}$ directions realise $(a,b)$.
For each direction, the walk is determined by the run-length composition
$\ell_1+\dots+\ell_{d-1}=n-1$ with each $\ell_i\ge1$, of which there are
$\binom{n-2}{d-2}$. Multiplying by the $1+\delta_{ab}$ directions gives~\eqref{eq:fclosed}.
\end{proof}

\begin{figure}[tbp]
\centering
\begin{tikzpicture}[x=0.75cm, y=0.5cm, font=\small]
  \begin{scope}
    \draw[->, gray!70!black] (0.5, 0) -- (8.7, 0);
    \node[font=\footnotesize, gray!70!black, anchor=west] at (8.7, 0) {$j$};
    \draw[->, gray!70!black] (0.5, -2.3) -- (0.5, 2.9);
    \node[font=\footnotesize, gray!70!black, anchor=south] at (0.5, 2.9) {$h$};
    \foreach \j in {1,...,8} {
      \draw[gray!70!black] (\j, 0.1) -- (\j, -0.1);
      \node[font=\footnotesize, gray!70!black, anchor=north] at (\j, -0.15) {\j};
    }
    \foreach \y in {-2,-1,1,2} {
      \draw[gray!70!black] (0.4, \y) -- (0.6, \y);
      \node[font=\footnotesize, gray!70!black, anchor=east] at (0.35, \y) {\y};
    }
    \node[font=\footnotesize, gray!70!black, anchor=east] at (0.35, 0) {0};
    \draw[decorate, decoration={brace, amplitude=3pt}, thick, gray!70!black]
      (1, 2.5) -- (3, 2.5) node[midway, above=1pt, font=\scriptsize] {$\ell_1{=}2$};
    \draw[decorate, decoration={brace, amplitude=3pt}, thick, gray!70!black]
      (3, 2.5) -- (6, 2.5) node[midway, above=1pt, font=\scriptsize] {$\ell_2{=}3$};
    \draw[decorate, decoration={brace, amplitude=3pt}, thick, gray!70!black]
      (6, 2.5) -- (8, 2.5) node[midway, above=1pt, font=\scriptsize] {$\ell_3{=}2$};
    \draw[very thick, black]
      (1,0) -- (2,1) -- (3,2) -- (4,1) -- (5,0) -- (6,-1) -- (7,0) -- (8,1);
    \foreach \x/\y/\h in {1/0/0, 6/-1/-1}
      \node[circle, draw=black, fill=gray!30, minimum size=6mm, inner sep=0pt] at (\x,\y) {$\h$};
    \foreach \x/\y/\h in {3/2/2, 8/1/1}
      \node[rectangle, draw=black, very thick, fill=white, minimum size=6mm, inner sep=0pt] at (\x,\y) {$\h$};
    \node[font=\footnotesize, anchor=north] at (4.5, -2.9) {(a) $h$, first run up};
  \end{scope}
  \begin{scope}[xshift=9.3cm]
    \draw[->, gray!70!black] (0.5, 0) -- (8.7, 0);
    \node[font=\footnotesize, gray!70!black, anchor=west] at (8.7, 0) {$j$};
    \draw[->, gray!70!black] (0.5, -2.3) -- (0.5, 2.9);
    \node[font=\footnotesize, gray!70!black, anchor=south] at (0.5, 2.9) {$h$};
    \foreach \j in {1,...,8} {
      \draw[gray!70!black] (\j, 0.1) -- (\j, -0.1);
      \node[font=\footnotesize, gray!70!black, anchor=north] at (\j, -0.15) {\j};
    }
    \foreach \y in {-2,-1,1,2} {
      \draw[gray!70!black] (0.4, \y) -- (0.6, \y);
      \node[font=\footnotesize, gray!70!black, anchor=east] at (0.35, \y) {\y};
    }
    \node[font=\footnotesize, gray!70!black, anchor=east] at (0.35, 0) {0};
    \draw[decorate, decoration={brace, amplitude=3pt}, thick, gray!70!black]
      (1, 2.5) -- (3, 2.5) node[midway, above=1pt, font=\scriptsize] {$\ell_1{=}2$};
    \draw[decorate, decoration={brace, amplitude=3pt}, thick, gray!70!black]
      (3, 2.5) -- (6, 2.5) node[midway, above=1pt, font=\scriptsize] {$\ell_2{=}3$};
    \draw[decorate, decoration={brace, amplitude=3pt}, thick, gray!70!black]
      (6, 2.5) -- (8, 2.5) node[midway, above=1pt, font=\scriptsize] {$\ell_3{=}2$};
    \draw[very thick, black]
      (1,0) -- (2,-1) -- (3,-2) -- (4,-1) -- (5,0) -- (6,1) -- (7,0) -- (8,-1);
    \foreach \x/\y/\h in {3/-2/-2, 8/-1/-1}
      \node[circle, draw=black, fill=gray!30, minimum size=6mm, inner sep=0pt] at (\x,\y) {$\h$};
    \foreach \x/\y/\h in {1/0/0, 6/1/1}
      \node[rectangle, draw=black, very thick, fill=white, minimum size=6mm, inner sep=0pt] at (\x,\y) {$\h$};
    \node[font=\footnotesize, anchor=north] at (4.5, -2.9) {(b) $-h$, first run down};
  \end{scope}
\end{tikzpicture}
\caption{Two $\pm1$ walks on $P_8$ with $(a,b)=(2,2)$, related by $h\mapsto-h$. Both share the run-length composition $(\ell_1,\ell_2,\ell_3)=(2,3,2)$ summing to $n-1=7$, with $r=d-1=3$ monotone runs. The reflection swaps the extremum pattern from min--max--min--max in (a) to max--min--max--min in (b), so both realise $a=b$. This pairing accounts for the factor $1+\delta_{ab}=2$ in~\eqref{eq:fclosed}, and $\binom{n-2}{d-2}=\binom{6}{2}=15$ compositions per direction give $f_{(2,2)}(8)=30$.}
\label{fig:walkcount}
\end{figure}

By Lemma~\ref{lem:binom}, $f_{(a,b)}$ vanishes identically unless $|a-b|\le1$, in
which case it is a polynomial in $n$ of degree exactly $d-2$ with leading
coefficient $(1+\delta_{ab})/(d-2)!$.

\begin{lemma}[Single-side decomposition]\label{lem:splitedge}
For each split $(a,b)$ with $d\ge3$ and $m,n\ge2$,
\[
  N_{(a,b)}(m,n)=2f_{(a,b)}(n)+2f_{(a,b)}(m)+R_{(a,b)}(m,n),
\]
where $2f_{(a,b)}(n)$ and $2f_{(a,b)}(m)$ count the four single-side families and
$R_{(a,b)}$ counts the configurations whose apexes are not all on one side, so
$R_{(a,b)}\ge0$. The single-side term is nonzero only when $|a-b|\le1$, in which case
its total-degree-$(d-2)$ part is
\[
  2f_{\mathrm{lead}}\,(m^{d-2}+n^{d-2}),\qquad f_{\mathrm{lead}}=\frac{1+\delta_{ab}}{(d-2)!}.
\]
\end{lemma}

\begin{proof}
The reduction of Lemma~\ref{lem:edgered} puts the top-row family in bijection with
the $\pm1$ walks on $P_n$ that have $a$ minima and $b$ maxima. A configuration
gives its walk $\tau$; conversely a walk with strict minima at
$\kappa_1<\dots<\kappa_a$ returns the configuration with apexes $(1,\kappa_s)$ and
offsets $\tau(\kappa_s)$. This configuration is valid because the unit steps of
$\tau$ give the parity condition and its strict minima the activity condition, which
is strict because a maximum separates any two minima. By Lemma~\ref{lem:edgered}
its envelope is $(i-1)+\tau(j)$, whose top row returns $\tau$. By the four-side
symmetry of Lemma~\ref{lem:edgered}, the top and bottom rows are each counted by
$f_{(a,b)}(n)$ and the left and right columns by $f_{(a,b)}(m)$
(Lemma~\ref{lem:binom}). Two families overlap only when some configuration has all
apexes on two distinct sides. Distinct sides meet in at most one cell, since opposite
sides share none and adjacent sides share a single corner. An overlap therefore forces
$a=1$ with the apex at that corner, whose unique farthest cell forces $b=1$. Hence
for $d\ge3$ the four families are disjoint, and $R_{(a,b)}$ is their complement. The count
$f_{(a,b)}$ vanishes unless $|a-b|\le1$. When $|a-b|\le1$ the top and
bottom rows contribute $2f_{(a,b)}(n)$, with top-degree part
$2f_{\mathrm{lead}}\,n^{d-2}$, and the left and right columns contribute $2f_{(a,b)}(m)$,
with top-degree part $2f_{\mathrm{lead}}\,m^{d-2}$. Each family depends on one axis only, so the top-degree part
carries no mixed monomial.
\end{proof}

% ===========================================================================
\section{The top-degree part}\label{sec:leading}
% REOPENED 2026-08-04: added lem:splitpoly and prop:converse (the converse of
% prop:leading, deriving conj:G2 at a given d from the top-degree part of E_d).
% Section 8 (sec:leading) LOCKED 2026-07-12 (Adversary drafts-review PASS: F1
% (subject-verb proximity), F2 (2 fronted commas By Prop/Under Conj), F3 (stacked
% frame comma) applied + verified correct/complete; N4/N5 left; prop:singleside
% math re-derived C(d)=4/(d-2)! (balanced split unique, (2-delta)(1+delta)=2 both
% parities), consistent with locked sec:period. Rubrics math-proofs, fit,
% english-usage, paper-writing, source-grounding, latex-conventions.)
% ===========================================================================

Section~\ref{sec:onedim} splits each count $N_{(a,b)}$ into a single-side term and the
remainder $R_{(a,b)}$ (Lemma~\ref{lem:splitedge}). Folding the single-side terms
over the splits gives their contribution to $E_d$.

\begin{proposition}[Single-side contribution]\label{prop:singleside}
For $d\ge3$ and $m,n\ge2$, the single-side contribution to the total-degree-$(d-2)$
part of $E_d(m,n)$ is
\[
  C(d)\,(m^{d-2}+n^{d-2}),\qquad C(d)=\frac{4}{(d-2)!}.
\]
Every remaining term of that part comes from the remainders $R_{(a,b)}$.
\end{proposition}

\begin{proof}
The count $E_d$ is the sum of the $N_{(a,b)}$ over the splits of $d$
with $a\le b$, each weighted by $2-\delta_{ab}$~\eqref{eq:Ed}. Lemma~\ref{lem:splitedge} splits each
$N_{(a,b)}$ into its single-side term and $R_{(a,b)}$. A single-side term is nonzero
only for $|a-b|\le1$ (Lemma~\ref{lem:binom}), so among the splits with $a\le b$ only the
balanced one, $a=\lfloor d/2\rfloor$ and $b=\lceil d/2\rceil$, contributes. Its
single-side term has top-degree part
\[
  2f_{\mathrm{lead}}\,(m^{d-2}+n^{d-2}),\qquad f_{\mathrm{lead}}=\frac{1+\delta_{ab}}{(d-2)!}.
\]
Weighting this by $2-\delta_{ab}$ gives the top-degree coefficient
\[
  (2-\delta_{ab})\cdot 2f_{\mathrm{lead}}
  =\frac{2\,(2-\delta_{ab})(1+\delta_{ab})}{(d-2)!}=\frac{4}{(d-2)!}.
\]
Since $(2-\delta_{ab})(1+\delta_{ab})=2$ for $\delta_{ab}\in\{0,1\}$, this is the same
for even and odd $d$. Each single-side monomial involves one axis, so the contribution
is $C(d)\,(m^{d-2}+n^{d-2})$. The remaining top-degree terms come from the weighted
remainders $R_{(a,b)}$.
\end{proof}

Speaking of the degree of an individual remainder presupposes that it has one.
The frozen-run contraction (Lemma~\ref{lem:contract}) and the interpolation of
Lemma~\ref{lem:interp} supply it, while Proposition~\ref{prop:converse} below rests on
the polynomiality of Theorem~\ref{thm:poly}. That polynomiality and both
lemmas are established in Section~\ref{sec:period} independently of the
following lemma and of that proposition, so no circularity arises.

\begin{lemma}[Per-split polynomiality]\label{lem:splitpoly}
Fix $d\ge3$ and a split $(a,b)$ with $a+b=d$ and $a\le b$. The count
$N_{(a,b)}$ agrees on $\{m,n\ge2d-1\}$
with a single bivariate polynomial of degree at most $d-2$ in each variable, and so does the remainder $R_{(a,b)}$.
\end{lemma}

\begin{proof}
The frozen-run contraction of Section~\ref{sec:period} (Lemma~\ref{lem:contract})
collapses each height function to a \emph{type} on at most $2d-1$ columns and
preserves every extremum with its sign. A type therefore determines the split of every
height function contracting to it, so the types of a fixed split can be counted on
their own. For large $n$, a type with $r\le d-1$ runs contributes a polynomial in $n$ of
degree $r-1\le d-2$ (Lemma~\ref{lem:contract}). The threshold beyond
which this holds is the type's column count, at most $2d-1$, uniform in $m$ and
over the finitely many types. Summing over the types of the split leaves a single
polynomial in $n$ of degree at most $d-2$. Transposing the grid exchanges the axes
while preserving the number of minima and the number of maxima, so the same bound
holds along columns. Lemma~\ref{lem:interp} assembles the rows and columns into
one bivariate polynomial of degree at most $d-2$ in each variable. The
single-side term is a polynomial by Lemma~\ref{lem:binom}, so subtracting it
leaves $R_{(a,b)}$. The region $\{m,n\ge2d-1\}$ is smaller than the sharp high
region $\{m,n\ge\max(d-1,2)\}$ of Theorem~\ref{thm:poly}. This weaker per-split
statement is all that Proposition~\ref{prop:converse} below requires.
\end{proof}

From configurations with all apexes collinear on one boundary line, the
single-side contribution reaches the full degree $d-2$. The remainder counts every other
configuration, and whether it too reaches $d-2$ is the question.

\begin{conjecture}[Degree bound]\label{conj:G2}
For $d\ge5$ and every split $(a,b)$ with $a+b=d$ and $a\le b$, the remainder $R_{(a,b)}$
has total degree at most $d-3$.
\end{conjecture}

\begin{proposition}[Top-degree part]\label{prop:leading}
For $d\ge5$, under Conjecture~\ref{conj:G2}, the total-degree-$(d-2)$ part of
$E_d(m,n)$ is $C(d)\,(m^{d-2}+n^{d-2})$.
\end{proposition}

\begin{proof}
By Proposition~\ref{prop:singleside}, the single-side terms supply
$C(d)\,(m^{d-2}+n^{d-2})$, and every other top-degree term comes from the
remainders $R_{(a,b)}$. Under Conjecture~\ref{conj:G2}, each $R_{(a,b)}$ has total
degree at most $d-3$, so no remainder contributes at degree $d-2$. The
top-degree part of $E_d$ is therefore the single-side contribution alone,
carrying no mixed monomial.
\end{proof}

The implication reverses. The top-degree part of $E_d$ is an object the
interpolation of Section~\ref{sec:data} can determine without assuming anything
about the individual remainders, and knowing it is enough to bound all of them at once.

\begin{proposition}[Degree bound from the top-degree part]\label{prop:converse}
Let $d\ge5$, and suppose that on the high region $E_d$ carries no monomial of
total degree above $d-2$ and that its total-degree-$(d-2)$ part is
$C(d)\,(m^{d-2}+n^{d-2})$. Then every remainder $R_{(a,b)}$ with $a+b=d$ and
$a\le b$ has total degree at most $d-3$, so Conjecture~\ref{conj:G2} holds at
that $d$.
\end{proposition}

\begin{proof}
Subtracting the single-side terms from~\eqref{eq:Ed} leaves
\[
  E_d-(\text{single-side terms})
  =\sum_{\substack{1\le a\le b\\ a+b=d}}\bigl(2-\delta_{ab}\bigr)R_{(a,b)}.
\]
This holds at every integer point with $m,n\ge2$. Once $m$ and $n$ are large, each side
agrees with a polynomial, the left by Theorem~\ref{thm:poly} and Lemma~\ref{lem:binom}
and the right by Lemma~\ref{lem:splitpoly}, so it is an identity of polynomials.
By Proposition~\ref{prop:singleside} the single-side contribution has
total-degree-$(d-2)$ part $C(d)\,(m^{d-2}+n^{d-2})$, and it carries nothing above
degree $d-2$, since each single-side term is a polynomial of degree $d-2$ in one
variable. By hypothesis $E_d$ has that same total-degree-$(d-2)$ part and nothing
above it either, so the two cancel and the displayed difference has total degree
at most $d-3$.

\medskip\noindent
Each remainder is a polynomial for large $m,n$ by Lemma~\ref{lem:splitpoly}, so
its homogeneous components are defined; write $R^{[k]}_{(a,b)}$ for the degree-$k$ one.
Suppose some remainder had total degree at least $d-2$, and let $D\ge d-2$ be the
largest total degree attained among the $R_{(a,b)}$. Each remainder counts
configurations (Lemma~\ref{lem:splitedge}), so $R_{(a,b)}\ge0$ at every integer
point of the high region. Fix positive integers $m_0,n_0$;
evaluating along the multiples $t\,(m_0,n_0)$ and dividing by $t^{D}$,
\[
  R^{[D]}_{(a,b)}(m_0,n_0)
  =\lim_{t\to\infty}\frac{R_{(a,b)}(t m_0,t n_0)}{t^{D}}\ \ge\ 0 ,
\]
so each $R^{[D]}_{(a,b)}$ is nonnegative on a dense set of directions of the
positive cone $m,n>0$, hence on the whole cone by continuity. The weights $2-\delta_{ab}$
lie in $\{1,2\}$, so $\sum(2-\delta_{ab})R^{[D]}_{(a,b)}$ is nonnegative on that
cone as well. It is not identically zero. Some $R^{[D]}_{(a,b)}$ is a nonzero
polynomial, so it is nonzero at some point of the open cone. Since it is nonnegative
there, it is strictly positive at that point, where every other summand is at least
zero. That sum is the degree-$D$ component of the displayed difference, which
therefore has total degree $D\ge d-2$, contradicting the previous paragraph.
\end{proof}

This isolates the top-degree law in the single degree bound of
Conjecture~\ref{conj:G2}. Proposition~\ref{prop:singleside} supplies
$C(d)\,(m^{d-2}+n^{d-2})$ unconditionally, leaving only the remainders in question.
The bound on them holds for the separable configurations, those whose envelope splits
into a row and a column part, by Section~\ref{sec:separable} (Theorem~\ref{thm:sepcase}).
Direct enumeration confirms it through $d=7$, and
Proposition~\ref{prop:converse} settles $d=8$ from the independently determined
$p_8$ (Corollary~\ref{cor:G2d8}). The converse needs $p_d$ pinned without
assuming the bound, which the interpolation of Section~\ref{sec:data} achieves
only through $d=8$. It reaches $p_9$ and $p_{10}$ under Conjecture~\ref{conj:G2}
itself, so the non-separable case stays open for $d\ge9$.

% ===========================================================================
\section{Polynomiality on the high region}\label{sec:period}
% REOPENED 2026-08-04: residual R~_d now unconditional for d<=8, was d<=7.
% Section 9 (sec:period) LOCKED 2026-07-11 (Adversary lock-gate PASS: 0 CRITICAL,
% 0 WRONG, 0 blocking; rubrics math-proofs, fit, english-usage, paper-writing,
% source-grounding, latex-conventions. Full internal correctness + completeness +
% cross-section linkage verified; quantitative claims brute-forced: E_2=4,
% p_2..p_7 match E_d, per-axis leading coeff C(d)=4/(d-2)!, tab:corr incl
% B_5^{(2)}=40, Baxter -4*Bax(d-3), d=8 remark. Consumers (abstract, intro,
% sec:data) match thm:poly; lem:uniform d>=3 never invoked at d=2; 20 refs
% resolve; H_d not orphaned. thm:poly reconciled to d>=2/max(d-1,2) (commit
% 8e0f5630); para-3 colon (61d8b4c7); M1 (E_2=4 direct enumeration) + M2
% (m^{d-2}-coeff precision) folded in. Word-level prose walk complete.)
% ===========================================================================

Sections~\ref{sec:onedim}--\ref{sec:leading} draw the top-degree
part of $E_d$ from the lattice-point count. Theorem~\ref{thm:quasipoly} leaves
open whether $E_d$ is a single polynomial of period $1$ rather than a genuine
quasi-polynomial, which the lattice-point count does not settle. The
column-to-column transfer matrix does. It establishes that $E_d$ is eventually polynomial in each
variable, with the onset uniform in $m$. $E_d$ therefore agrees with a single polynomial on the rectangle $H_d=\{m,n\ge d-1\}$
for $d\ge3$. The question reduces to a per-axis statement, resolved by classifying
the columns that carry no extremum.

\subsection{Per-axis polynomiality}

With $m$ fixed, a height function is a left-to-right sequence of columns, each assigning integer
heights to the $m$ cells of the path $P_m$. Heights differ by $\pm1$ across each edge, so mod $3$
each column becomes a proper $3$-colouring, an element $\sigma\in\{0,1,2\}^m$ with adjacent entries
distinct. A transition $\sigma\to\sigma'$ is admissible iff $\sigma'_i\ne\sigma_i$
for all $i$. A
cell is a strict local extremum iff its neighbours share one colour, a
test spanning three consecutive columns. The bulk transfer
matrix therefore acts on admissible pairs of consecutive columns. It steps from a
pair $(\text{prev},\text{cur})$ to $(\text{cur},\text{next})$, emitting
$x^{\#\text{extrema of cur}}$ for the middle column of the triple
$(\text{prev},\text{cur},\text{next})$. Let $T_m(x)$ be this finite matrix over $\Z[x]$, indexed by the admissible pairs.

\begin{lemma}[Rational generating function]\label{lem:ratGF}
\[
  \sum_{n\ge2} c_{m,n}(x)\,z^n=z^2\,\mathbf u^{\!\top}\big(I-zT_m(x)\big)^{-1}\mathbf v,
\]
where $c_{m,n}(x)=\sum_h x^{\#\text{extrema}(h)}$ sums over the height functions $h$
on the $m\times n$ grid and $\mathbf u,\mathbf v$ are boundary vectors polynomial in
$x$ and independent of $z$. Hence $\sum_{n\ge2} E_d(m,n)\,z^n$ is the coefficient of $x^d$ in the
right-hand side, and so rational in~$z$.
\end{lemma}

\begin{proof}
A height function on the $m\times n$ grid is a sequence of $n$ columns, each a
proper $3$-colouring of $P_m$ with consecutive columns admissible. Whether a cell is
a strict local extremum depends only on its column and the two adjacent columns, so
$\#\text{extrema}(h)$ splits into per-column contributions, each from a consecutive
triple. The matrix $T_m(x)$ weights each transition by its middle column's
contribution, so $x^{\#\text{extrema}(h)}$ factors as a product along the pair
sequence. Summing over all height functions gives
\[
  c_{m,n}(x)=\mathbf u^{\!\top}T_m(x)^{\,n-2}\mathbf v,
\]
where $\mathbf u,\mathbf v$ supply the two end columns and $\mathbf u$ fixes the
corner $h(1,1)=0$, so the count runs over height functions, not all colourings.
Summing the geometric series in $z$ gives the stated identity, the factor $z^2$
offsetting $n$ against the power $n-2$. Finally, the resolvent
$(I-zT_m(x))^{-1}=\operatorname{adj}(I-zT_m(x))/\det(I-zT_m(x))$ has entries that are
ratios of polynomials in $\Z[x,z]$ whose denominator does not vanish at $x=0$. Expanding in $x$
gives coefficients rational in $z$, so the coefficient of $x^d$ in the right-hand
side is rational in $z$.
\end{proof}

\begin{lemma}[Pole location]\label{lem:poles}
Let $T_0:=T_m(0)$ be the extremum-free transfer matrix. The poles of
$\sum_{n\ge2} E_d(m,n)\,z^n$ lie among the reciprocals of the nonzero eigenvalues
of $T_0$.
\end{lemma}

\begin{proof}
By Lemma~\ref{lem:ratGF}, $\sum_{n\ge2} E_d(m,n)\,z^n$ is $z^2$ times the $x^d$-coefficient of
$\mathbf u^{\!\top}(I-zT_m(x))^{-1}\mathbf v$, with $\mathbf u,\mathbf v$ polynomial in $x$. Its poles
are therefore among those of this coefficient, read off the resolvent at $x=0$,
which the extrema only perturb. At $x=0$, the resolvent is $R_0:=(I-zT_0)^{-1}$. Its poles are the
roots of $\det(I-zT_0)=\prod_i(1-z\lambda_i)$, the reciprocals $1/\lambda_i$ of the nonzero
eigenvalues of $T_0$. Writing $T_m(x)=T_0+xB(x)$ separates off the extrema, and the full resolvent
is the Neumann series
\[
  (I-zT_m(x))^{-1}=\sum_{k\ge0}\big(R_0\,zxB(x)\big)^k R_0.
\]
Each factor $zxB(x)$ carries a positive power of $x$, so the $k$-th term has $x$-degree at least
$k$. Only those with $k\le d$ contribute to $[x^d]$. The $x^d$-coefficient is thus a finite sum of
products of $R_0$, with no denominator beyond $\det(I-zT_0)$. The extrema therefore add no poles,
and every pole is one of the $1/\lambda_i$.
\end{proof}

\begin{lemma}[Frozen classification]\label{lem:frozen}
Let $a,b,c$ be three consecutive columns of a height function, so each is a proper
$3$-colouring of $P_m$, with $a_i\ne b_i$ and $c_i\ne b_i$ for all $i$. The middle
column $b$ carries no strict local extremum iff there is a slope
$k\in\{1,2\}$ with $a_i=b_i-k$ and $c_i=b_i+k$ in $\Z/3$ for all $i$. Such
a triple is \emph{frozen}.
\end{lemma}

\begin{proof}
Index the rows $0,\dots,m-1$ and work in $\Z/3$. Let $\alpha_i=a_i-b_i$,
$\gamma_i=c_i-b_i$, and $s_i=b_{i+1}-b_i$. These are nonzero, hence in $\{1,2\}$,
because $b$ is proper and $a_i\ne b_i$, $c_i\ne b_i$. Since $a_{i+1}\ne a_i$ and
$c_{i+1}\ne c_i$,
\begin{equation}\label{eq:uwproper}
  \alpha_{i+1}\ne\alpha_i-s_i,\qquad \gamma_{i+1}\ne\gamma_i-s_i.
\end{equation}
Row $i$ is \emph{aligned}, with value $r:=\alpha_i$, when $\alpha_i=\gamma_i$, and
\emph{rainbow} when $\alpha_i\ne\gamma_i$. A rainbow row shows $a_i,b_i,c_i$ in all
three colours, so $\gamma_i=3-\alpha_i$.

\medskip\noindent
Each present neighbour of a cell differs from it by $\pm1$, so the cell is a strict
local extremum exactly when its neighbours share a single colour. For the cell in
row $i$ of $b$, the horizontal neighbours $a_i$ and $c_i$ share the colour
$b_i+\alpha_i$ precisely when the row is aligned. The up-neighbour
$b_{i-1}=b_i-s_{i-1}$, present for $i>0$, carries that colour when
$s_{i-1}=3-\alpha_i$. The down-neighbour $b_{i+1}=b_i+s_i$, present for $i<m-1$,
does so when $s_i=\alpha_i$. An extremum therefore occurs only on an aligned row,
and an aligned row $i$ carries one exactly when
\begin{equation}\label{eq:extremum}
  \begin{cases}
    s_{i-1}=3-\alpha_i \text{ and } s_i=\alpha_i, & 0<i<m-1,\\
    s_0=\alpha_0, & i=0,\\
    s_{m-2}=3-\alpha_{m-1}, & i=m-1.
  \end{cases}
\end{equation}

\medskip\noindent
Alignment propagates along $b$. If row $i$ is aligned with value $r$ and
$s_i=3-r$, then \eqref{eq:uwproper} leaves $\alpha_{i+1}$ and $\gamma_{i+1}$
different from $3-r$. Both therefore equal $r$, so row $i+1$ is aligned with value
$r$. Symmetrically, $s_{i-1}=r$ aligns row $i-1$ with value $r$.

\medskip\noindent
Suppose first that $b$ is extremum-free. If some row were aligned, take a maximal
block $[p,q]$ of consecutive rows aligned with a single value $r$. Maximality pins
the two end steps. If $q<m-1$ then $s_q=r$, since $s_q=3-r$ would align row $q+1$.
If $p>0$ then $s_{p-1}=3-r$, since $s_{p-1}=r$ would align row $p-1$. No aligned row
is an extremum, so by \eqref{eq:extremum} an interior row $j$ with $s_{j-1}=3-r$ has
$s_j=3-r$. The value $3-r$ enters at the lower end: as $s_{p-1}$ when $p>0$, and as
$s_0$ when $p=0$, since the non-extremal row $0$ makes $s_0=3-r$. It propagates up
to $s_{q-1}$. When $q<m-1$ it reaches $s_q=3-r$ as well, through interior row $q$
for $q>0$ and as $s_0$ for $q=0$, contradicting $s_q=r$. When $q=m-1$, the value
$s_{m-2}=3-r$ makes row $m-1$ an extremum. Both cases are impossible, so no row is
aligned, and every row is rainbow with $\gamma_i=3-\alpha_i$. The $\gamma$-clause of
\eqref{eq:uwproper} now reads $\alpha_{i+1}\ne\alpha_i+s_i$, which with
$\alpha_{i+1}\ne\alpha_i-s_i$ forces $\alpha_{i+1}=\alpha_i$. Hence $\alpha$ is the
constant $-k$ for some $k\in\{1,2\}$, and $\gamma=3-\alpha=k$. Then $a=b-k$ and
$c=b+k$, so the middle column is frozen.

\medskip\noindent
Conversely, suppose $a=b-k$ and $c=b+k$ with $k\in\{1,2\}$. Then $\alpha_i=-k$ and
$\gamma_i=k$ at every row, and $-k\ne k$ in $\Z/3$, so no row is aligned. An
extremum occurs only on an aligned row, so $b$ is extremum-free.
\end{proof}

By contrast, a boundary column has only one horizontal neighbour, and so always
carries an extremum.

\begin{lemma}[Boundary extremum]\label{lem:boundary}
In every height function on $\Gmn$ with $m,n\ge2$, the first and last columns each
carry a strict local extremum.
\end{lemma}

\begin{proof}
It suffices to treat the first column, the last being symmetric. Its heights
$h_0,\dots,h_{m-1}$ form a $\pm1$ walk with steps $\varepsilon_i:=h_{i+1}-h_i$. The
second column's heights $h'_0,\dots,h'_{m-1}$ form a $\pm1$ walk too, at offset
$\eta_i:=h'_i-h_i\in\{\pm1\}$ from the first.

\medskip\noindent
A cell of the first column is a strict local extremum when its present neighbours
all lie above it or all below. Each differs from $h_i$ by $1$, so cell $i$ is an
extremum exactly when
\begin{equation}\label{eq:boundchar}
  \begin{cases}
    \varepsilon_0=\eta_0, & i=0,\\
    -\varepsilon_{i-1}=\varepsilon_i=\eta_i, & 0<i<m-1,\\
    \varepsilon_{m-2}=-\eta_{m-1}, & i=m-1.
  \end{cases}
\end{equation}

\medskip\noindent
Because the second column is a $\pm1$ walk,
$h'_{i+1}-h'_i=\varepsilon_i+(\eta_{i+1}-\eta_i)$ lies in $\{\pm1\}$, forcing
$\varepsilon_i=\eta_i$ wherever $\eta$ changes. Step $i$ is an \emph{ascent} when
$\varepsilon_i=+1$ and a \emph{descent} when $\varepsilon_i=-1$. Then $\eta$ changes from
$-1$ to $+1$ only at a descent, and from $+1$ to $-1$ only at an ascent.

\medskip\noindent
Suppose no cell is an extremum. Then $\eta$ is constant. Otherwise, by symmetry,
$\eta$ changes from $+1$ to $-1$ at some step $i$, an ascent. A rise back to $+1$
would need a descent, which \eqref{eq:boundchar} forbids at the interior cells while
$\eta=-1$, so $\varepsilon_j=+1$ and $\eta_j=-1$ for all $j>i$. Then
$\varepsilon_{m-2}=+1=-\eta_{m-1}$ makes cell $m-1$ an extremum, a contradiction.
With $\eta$ constant at $\eta^\ast$, cell $0$ forces $\varepsilon_0=-\eta^\ast$, and
the interior conditions carry $\varepsilon_i=-\eta^\ast$ up the column. Then
$\varepsilon_{m-2}=-\eta^\ast=-\eta_{m-1}$ makes cell $m-1$ an extremum, again a
contradiction. Hence some cell is an extremum.
\end{proof}

\begin{lemma}[Rotation quotient]\label{lem:quotient}
For every $m\ge2$, the generating function $\sum_{n\ge2}E_d(m,n)\,z^n$ has its only
pole at $z=1$, so $E_d(m,\cdot)$ is eventually a polynomial in $n$.
\end{lemma}

\begin{proof}
The rotation $\rho:\sigma\mapsto\sigma+1$ on $\Z/3$ preserves admissibility and the
extremum count, so it commutes with $T_m(x)$ for every $x$. It acts freely on
admissible pairs, since $(\sigma+1,\sigma'+1)=(\sigma,\sigma')$ is impossible. Fixing
the corner picks one representative per rotation orbit. Since the extremum count is
$\rho$-invariant, summing over these representatives averages over each orbit,
which projects the boundary vectors onto the $\rho$-invariant subspace. As $T_0$
commutes with that projection, the count sees only $T_0^{\mathrm{triv}}$, the action
of $T_0$ on $\rho$-invariant states.

\medskip\noindent
By Lemma~\ref{lem:frozen}, a pair $(\mathrm{prev},\mathrm{cur})$ takes an
extremum-free step to $(\mathrm{cur},\mathrm{next})$ exactly when the triple is
frozen of some slope $k$. Then $\mathrm{prev}=\mathrm{cur}-k$ and
$\mathrm{next}=\mathrm{cur}+k$. The states on cycles of $T_0$ are therefore the
frozen pairs $(\sigma-k,\sigma)$, each stepping to its unique successor
$(\sigma,\sigma+k)$. Every other pair has no extremum-free edge in or out, so it is
transient, of eigenvalue $0$. A frozen pair advances $\sigma$ by $k$ at each step
and returns after three, since $3k\equiv0$, so it lies on a $3$-cycle. Each
$3$-cycle contributes eigenvalues $\{1,\omega,\omega^2\}$, with
$\omega=e^{2\pi i/3}$, and these make up the unit-circle spectrum of $T_0$. The
$\omega$ and $\omega^2$ modes are carried by the rotation.

\medskip\noindent
The successor satisfies $(\sigma,\sigma+k)=\rho^{\,k}(\sigma-k,\sigma)$, so the
transfer step acts as $\rho^{\,k}$ on each $3$-cycle. On the $\rho$-invariant
subspace, $\rho$ acts as the identity, so each $3$-cycle collapses to a self-loop. Hence
$T_0^{\mathrm{triv}}$ is the identity on the frozen orbits and zero on the
transient ones, with spectrum $\{0,1\}$. The count is $\rho$-invariant, so only the
block $T_0^{\mathrm{triv}}$ contributes to its poles, and there the only nonzero
eigenvalue is $1$. Lemma~\ref{lem:poles}, applied to this block, therefore leaves
$z=1$ as the sole pole of $\sum_{n\ge2} E_d(m,n)\,z^n$. A rational function with its only pole at $z=1$ has
eventually-polynomial coefficients, so $E_d(m,\cdot)$ is eventually a polynomial in
$n$.
\end{proof}

The onset of this polynomial form may still grow with $m$. Lemma~\ref{lem:uniform}
below bounds the onset at $d-1$ for $d\ge3$ and the per-axis degree at $d-2$,
uniformly in $m$. This uniformity turns per-axis polynomiality into agreement on a
genuine rectangle.

Let $\Pi:=T_0^{\mathrm{triv}}$ be the projection the count sees
(Lemma~\ref{lem:quotient}), and $Q:=I-\Pi$. The boundary
vectors of Lemma~\ref{lem:ratGF} split by the extremum count of the boundary column they
carry, $\mathbf u=\sum_{i\ge1}x^i\mathbf u_i$ and
$\mathbf v=\sum_{i\ge1}x^i\mathbf v_i$. The sum starts at $i=1$ because
Lemma~\ref{lem:boundary} makes every boundary column carry an extremum. The
$\mathbf u_i$ and $\mathbf v_i$ do not depend on $z$. The components $\mathbf u_1$ and
$\mathbf v_1$ record a boundary column with a single extremum. The sharp onset
rests on a fixed-point identity for these, one that removes the single-extremum
boundary terms and so lowers the numerator degree bound in
Lemma~\ref{lem:uniform}, sharpening its onset.

\begin{lemma}[Frozen single-extremum boundary]\label{lem:frozenbdy}
On the count block, $\mathbf u_1^{\!\top}Q=0$ and $Q\mathbf v_1=0$.
\end{lemma}

\begin{proof}
By left--right symmetry, it suffices to prove $\mathbf u_1^{\!\top}T_0=\mathbf u_1^{\!\top}$ on the
count block. Its mirror identity $T_0\mathbf v_1=\mathbf v_1$
then holds, and $\mathbf u_1^{\!\top}Q=0$ and $Q\mathbf v_1=0$ follow. Written on
the states, $\mathbf u_1^{\!\top}T_0=\mathbf u_1^{\!\top}$ concerns the columns $a$ that may sit to
the left of $b$. A column $a$ is \emph{valid} when it is admissible with $b$, carries one
extremum with right neighbour $b$, and leaves $b$ extremum-free in the triple
$(a,b,c)$. The claim is that for every admissible pair $(b,c)$,
\begin{equation}\label{eq:frozenid}
  \#\{\text{valid }a\}=[\,b\ \text{has one extremum with right neighbour }c\,].
\end{equation}
A left neighbour can only destroy extrema of $b$, since an extremum needs its
present neighbours to share a colour and a new neighbour may break that. So a valid
$a$ takes the third colour at each extremum cell of $b$, lying above $b$ at a
maximum and below at a minimum.

\medskip\noindent
The columns $a$ and $b$ lift to height walks $\alpha,\beta$ as in the proof of
Lemma~\ref{lem:boundary}, with steps $\varepsilon_i:=\alpha_{i+1}-\alpha_i$ and
offset $\eta_i:=\beta_i-\alpha_i$, both in $\{\pm1\}$, and $\beta$ again a $\pm1$
walk. That proof's flip rule lets $\eta$ change from $-1$ to $+1$
only on a descent, and from $+1$ to $-1$ only on an ascent.

\medskip\noindent\emph{Every valid $a$ is a uniform shift.} A valid $a$ has one
extremum; say a maximum at $p$, so $\eta_p=-1$. Then $\alpha$ has a single local
maximum. Suppose instead a second maximum $p''>p$, with a local minimum $\mu$
between them; the case $p''<p$ is symmetric. On the descent $[p,\mu]$, the flip
rule bars every $\eta$-change $+1\to-1$. So either $\eta$ reaches $+1$ before $\mu$,
making the walk minimum $\mu$ an extremum of $a$, as $\eta_\mu=+1$ there; or $\eta$
stays $-1$ through $\mu$, and on the ascent $[\mu,p'']$ the flip rule bars
$-1\to+1$, so $\eta_{p''}=-1$ makes $p''$ an extremum. Either way, $a$ has a second
extremum, against validity. So $\alpha$ is unimodal, with $\varepsilon_i=+1$ for
$i<p$ and $-1$ for $i\ge p$. Then the increment
\[
  \eta_{i+1}-\eta_i=(\beta_{i+1}-\beta_i)-\varepsilon_i
\]
is $\le0$ for $i<p$ and $\ge0$ for $i\ge p$, so $\eta$ falls to $\eta_p=-1$ at the peak and rises after. It
remains to pin both corners at $-1$. Consider cell $0$. If $p=0$, then
$\eta_0=\eta_p=-1$. Otherwise, $p>0$ gives $\varepsilon_0=+1$ by unimodality, and
cell $0$ is not an extremum, so the test $\varepsilon_0=\eta_0$ of
Lemma~\ref{lem:boundary} fails, forcing $\eta_0=-1$. Cell $m-1$ is the same. If
$p=m-1$, then $\eta_{m-1}=-1$. Otherwise, $\varepsilon_{m-2}=-1$ and cell $m-1$ is
not an extremum, so $\varepsilon_{m-2}=-\eta_{m-1}$ fails, forcing $\eta_{m-1}=-1$.
A $\pm1$ sequence non-increasing to $-1$ then non-decreasing back
to $-1$ is constant, so $\eta\equiv-1$ and $a=b+1$. A minimum gives $a=b-1$.

\medskip\noindent\emph{The shift count.} The column $b+1$ lies above $b$, so it
destroys exactly the maxima of $b$ and leaves its minima. Its own extrema against
$b$ are the local maxima of $\beta$, endpoints included and independent of $c$;
likewise, $b-1$ destroys the minima and has extrema at the local minima. A shift is
valid only if it destroys every extremum of $b$, so $b+1$ can be valid only when
$b$ has no minimum, and $b-1$ only when $b$ has no maximum. By
Lemma~\ref{lem:boundary}, $b$ carries at least one extremum. Suppose it carries
exactly one, a maximum. Then $b$ has no minimum, and the same unimodality applied
to $b$ against $c$ gives $\beta$ a single local maximum; so $b+1$ has one extremum
and is valid, while $b-1$ leaves that maximum and is not. A single minimum is
symmetric. Suppose instead $b$ carries two or more extrema. Then either it has both
a maximum and a minimum, so neither shift destroys all of them; or it has two of
one type, say two maxima, so $b+1$ inherits two extrema from the two local maxima
of $\beta$. Either way, no shift is at once valid and single-extremum. The valid
columns therefore number $[\,b\ \text{has one extremum with right neighbour }c\,]$,
which is~\eqref{eq:frozenid}.
\end{proof}

The per-axis degree rests on a separate contraction, of frozen runs rather than
boundary terms.

\begin{lemma}[Frozen-run contraction]\label{lem:contract}
Fix $d\ge3$ and $m\ge2$. Contracting every maximal frozen run of columns of a
height function on $\Gmn$ with $d$ extrema to a single column preserves every
adjacent height difference, hence every extremum together with its sign, and
yields a height function on a grid of at most $2d-1$ columns, its \emph{type}.
For fixed $m$ the types form a finite set, and a type with $r\le d-1$ runs is
recovered on $\Gmn$ by extending each run to length at least one; for large
$n$ its preimages number a polynomial in $n$ of degree $r-1\le d-2$.
\end{lemma}

\begin{proof}
A column is \emph{occupied} when it carries an extremum. By
Lemma~\ref{lem:boundary} the two boundary columns are occupied, and by
Lemma~\ref{lem:frozen} a middle column is occupied exactly when its triple is not
frozen, so the $d$ extrema lie in at most $d$ occupied columns. A maximal run of
consecutive frozen columns lies strictly between two occupied columns and, by
Lemma~\ref{lem:frozen}, has a single slope in $\{1,2\}$; its columns step by that
slope from the first, so the run is fixed by its first column, its slope, and its
length. Each run lies between two consecutive occupied columns, so the runs number
at most one fewer than the occupied columns, hence at most $d-1$. Contracting every maximal frozen run to a single column and shifting
the later columns to match preserves every adjacent height difference and so
every extremum together with its sign. The image has at
most $2d-1$ columns. For fixed $m$ the types form a finite set, and a type with
$r$ runs is recovered by extending each run to length at least one, a count
polynomial in $n$ of degree $r-1$ for large $n$.
\end{proof}

\begin{lemma}[Uniform onset and degree]\label{lem:uniform}
For every $d\ge3$ and $m\ge2$, the count $E_d(m,\cdot)$ agrees on
$\{n\ge d-1\}$ with a single polynomial in $n$ of degree at most $d-2$.
\end{lemma}

\begin{proof}
The degree bound follows from the frozen-run contraction of
Lemma~\ref{lem:contract}, the sharper onset from the generating function. For the
degree: wherever $E_d(m,\cdot)$ agrees with a polynomial in $n$, that polynomial
has degree at most $d-2$, since a height function with $d$ extrema contracts to
one of finitely many types on at most $2d-1$ columns and a type with $r\le d-1$
runs is recovered as a polynomial in $n$ of degree $r-1\le d-2$. The onset $n\ge d-1$
comes from bounding the numerator degree of $\sum_{n\ge2}E_d(m,n)\,z^n$, with
Lemma~\ref{lem:frozenbdy} shaving the single-extremum boundary terms.

\medskip\noindent By Lemma~\ref{lem:ratGF}, the generating function
is
\[
  \sum_{n\ge2}E_d(m,n)\,z^n=z^2\,[x^d]\,\mathbf u^{\!\top}(I-zT_m(x))^{-1}\mathbf v .
\]
Split $T_m(x)=T_0+\sum_{j\ge1}x^jB_j$, where $B_j$ marks a middle column carrying
$j$ extrema. On the count block $T_0$ is the idempotent $\Pi$ of
Lemma~\ref{lem:quotient}; expanding the resolvent about $T_0$ gives
\[
  (I-zT_m(x))^{-1}=\sum_{k\ge0}R_0\Big(z\textstyle\sum_{j\ge1}x^jB_j\,R_0\Big)^k,
  \qquad R_0:=(I-zT_0)^{-1}=\frac{I-zQ}{1-z},
\]
the closed form following from $\Pi^2=\Pi$ and $Q=I-\Pi$.

\smallskip\noindent
A term of the $k$-th summand carries $\mathbf u_i,\mathbf v_{i'}$ and factors
$B_{j_1},\dots,B_{j_k}$, reaching $[x^d]$ only when $i+i'+j_1+\dots+j_k=d$. Since
$i,i'\ge1$ and each $j_\ell\ge1$, the interior extrema $p:=\sum_\ell j_\ell$
number at most $d-2$, and $k\le p$. Its only pole is at $z=1$, of order at most
$k+1\le d-1$. With the generating function written as $M(z)/(1-z)^{d-1}$, this term
contributes
\[
  z^{\,2+k}(1-z)^{d-2-k}\,
  \mathbf u_i^{\!\top}(I-zQ)B_{j_1}(I-zQ)\cdots B_{j_k}(I-zQ)\,\mathbf v_{i'}
\]
to $M$. Its $z$-degree is at most $(2+k)+(d-2-k)+(k+1)=d+k+1$: the $z^{2+k}$
prefactor, the $(1-z)^{d-2-k}$ power, and the $k+1$ resolvent factors in turn.

\smallskip\noindent
Boundary shaving lowers this bound. By Lemma~\ref{lem:frozenbdy}, a unit index
removes one factor, giving $\mathbf u_1^{\!\top}(I-zQ)=\mathbf u_1^{\!\top}$ when $i=1$
and $(I-zQ)\mathbf v_1=\mathbf v_1$ when $i'=1$. When $i=i'=1$ these shave two
distinct factors, since $p=d-2\ge1$ forces $k\ge1$ and hence at least two
$(I-zQ)$ factors, so both may be removed. Let $s\in\{0,1,2\}$ be the
number of unit indices. Removing $s$ factors drops the degree by $s$, while
$i+i'\ge4-s$ bounds $k\le p\le d-4+s$. The two effects cancel:
\[
  (d+k+1)-s \;\le\; d+(d-4+s)+1-s \;=\; 2d-3 .
\]
Every term therefore has degree at most $2d-3$, so $\deg M\le2d-3$. The
coefficients of $M(z)/(1-z)^{d-1}$ then agree with a single polynomial in $n$
for $n>(2d-3)-(d-1)=d-2$, that is for $n\ge d-1$.
\end{proof}

\subsection{Bivariate interpolation}

The count $E_d$ is polynomial in each variable separately
(Lemma~\ref{lem:uniform}). The following lemma assembles these one-variable
polynomials into a single bivariate polynomial.

\begin{lemma}[Bivariate interpolation]\label{lem:interp}
Let $N\ge1$ and $D\ge0$. If a function $F:\{N,N+1,\dots\}^2\to\Q$ has row
functions $F(m_0,\cdot)$ and column functions $F(\cdot,n_0)$ that are all
polynomials of degree at most $D$, then $F$ agrees on $\{m,n\ge N\}$ with a
single bivariate polynomial of degree at most $D$ in each variable.
\end{lemma}

\begin{proof}
Let $m_0$ range over the $D+1$ consecutive integers $N,N+1,\dots,N+D$. Let
$g_{m_0}(n):=F(m_0,n)$ be the corresponding row polynomials, each of degree at
most $D$. Set
\[
  P(m,n)=\sum_{m_0}g_{m_0}(n)\,L_{m_0}(m),
\]
where $L_{m_0}$ is the degree-$D$ Lagrange basis polynomial for these nodes. Then
$P$ has degree at most $D$ in each variable. For each fixed $n_0$, the map
$m\mapsto P(m,n_0)$ is the degree-$\le D$ interpolant of the values $F(m_0,n_0)$
at the nodes, and $m\mapsto F(m,n_0)$ is itself a polynomial of degree at most
$D$. Two such polynomials agreeing at $D+1$ nodes coincide, so
$P(m,n_0)=F(m,n_0)$ for all $m\ge N$. Since $n_0$ is arbitrary, $P\equiv F$.
\end{proof}

\begin{theorem}[Polynomiality on the high region]\label{thm:poly}
For each $d\ge2$, $E_d$ agrees on $\{m,n\ge\max(d-1,2)\}$ with a single symmetric
bivariate polynomial $p_d(m,n)$ of degree exactly $d-2$ in each variable. The
existence, symmetry, region, and per-axis degree are unconditional. The total
degree of $p_d$ is exactly $d-2$, unconditionally for $d\le8$ and under the
degree bound of Conjecture~\ref{conj:G2} for $d\ge9$.
\end{theorem}

\begin{proof}
For $d=2$, direct enumeration gives the constant $E_2\equiv4$ on
$\{m,n\ge2\}=\{m,n\ge\max(d-1,2)\}$; the constant $p_2\equiv4$ is symmetric with
per-axis and total degree $0=d-2$, all unconditional. Assume henceforth
$d\ge3$, so that $\max(d-1,2)=d-1$.

\medskip\noindent
The count $E_d(m,\cdot)$ agrees on $\{n\ge d-1\}$ with a polynomial in $n$ of
degree at most $d-2$, uniform in $m$ (Lemma~\ref{lem:uniform}). Transposition
$\Gmn\cong G_{n,m}$ preserves height functions and their extrema, so
$E_d(\cdot,n)$ is dually a polynomial in $m$ of degree at most $d-2$ on
$\{m\ge d-1\}$ for every $n$, and $E_d$ is symmetric. Bivariate interpolation
with threshold $d-1$ and degree $d-2$ (Lemma~\ref{lem:interp}) yields a
polynomial $p_d$ of degree at most $d-2$ in each variable agreeing with $E_d$ on
$H_d$. Both $p_d(m,n)$ and $p_d(n,m)$ agree with the symmetric $E_d$ on the symmetric
region $H_d$, and two polynomials agreeing on an infinite rectangle coincide, so
$p_d(m,n)=p_d(n,m)$. Below $H_d$, the count is recorded by
explicit corrections in Section~\ref{sec:data}.

\medskip\noindent
The per-axis degree is exactly $d-2$. The single-side count
(Lemma~\ref{lem:splitedge}) contributes a pure term
$C(d)\,m^{d-2}$ with $C(d)=4/(d-2)!>0$ (Proposition~\ref{prop:singleside}). The remainder has $m$-degree at most
$d-2$ and is a nonnegative count, so for each fixed $n$ it is a polynomial in $m$
nonnegative on $\{m\ge d-1\}$. Its coefficient of $m^{d-2}$ is therefore nonnegative,
equal to the leading coefficient when the degree is $d-2$ and zero otherwise. The $m^{d-2}$-coefficient of $E_d$ is
thus at least $C(d)>0$, so with the upper bound $d-2$ the degree in $m$ is
exactly $d-2$. By symmetry, so is the degree in $n$.

\medskip\noindent
With the per-axis degree $d-2$, the total degree is at most $2(d-2)$ a priori.
For $d=3,4$, direct computation gives total degree $d-2$. For $d\ge5$,
Proposition~\ref{prop:leading} identifies total degree $d-2$ with the degree bound
of Conjecture~\ref{conj:G2}. The single-side families supply the top-degree part
$C(d)(m^{d-2}+n^{d-2})$, and the bound rules out every monomial of higher total
degree. That bound holds unconditionally through $d=8$, and for the separable configurations at every $d$
(Theorem~\ref{thm:sepcase}); only the non-separable case at $d\ge9$ is open. Through
$d=8$ the bound is unconditional by direct enumeration for $d\le7$ and, at $d=8$, from
the independently determined $p_8$ (Corollary~\ref{cor:G2d8}). These certifications rest on the per-axis
polynomiality established above, not on the present total-degree claim, so the
total degree is $d-2$ unconditionally for $d\le8$ and under
Conjecture~\ref{conj:G2} for $d\ge9$.
\end{proof}

The interpolation route makes period $1$ a consequence of per-axis
polynomiality, as the row polynomials carry no period, so their Lagrange
combination carries none. The apparent period-$3$ behaviour seen before the quotient is
exactly the colour-rotation artefact removed by passing to $\OFG$.

% ===========================================================================
\section{The separable case of the degree bound}\label{sec:separable}
% Section 10 (sec:separable) LOCKED 2026-07-12 (holistic Adversary LOCK-WHOLE:
% 0 CRITICAL/WRONG, no blocking. Intro conj:G2 restatement faithful; arc
% intro -> lem:sep -> lem:dimid -> thm:sepcase -> residual -> d=7 certification
% complete + non-redundant (the d=4 break appears in BOTH subsections but as
% DISTINCT separable/non-separable families; "again" flags the parallel).
% Cross-section linkage confirmed: thm:sepcase supports the sec:leading
% back-citation; conj:G2 one meaning throughout; the partitions N=single-side+R
% (sec:leading) and N=N^s+N^ns (sec:separable) are bridged (single-side within
% separable); the six feeder-lemma restatements match; d<=7-certified /
% d>=8-open consistent across abstract/intro/sec:leading/sec:period; the split
% list is complete (a=1 always separable). Notation/sort clean section-wide
% (10.1 D1 + 10.2 E3 hold). Subsections locked: 10.1 = 494458c8, 10.2 = 3f72d6f1.
% Rubrics math-proofs, fit, english-usage, paper-writing, source-grounding, latex.)
% ===========================================================================

The total degree $d-2$ of $p_d$ rests on the
degree bound of Section~\ref{sec:leading} (Conjecture~\ref{conj:G2}), the
assertion that no remainder $R_{(a,b)}$ reaches the degree attained by the
single-side families. The bound is proved here for the separable case,
leaving the non-separable residual as the sole gap.

\subsection{The separable configurations}
% Subsection 10.1 (separable configurations) LOCKED 2026-07-12 (Adversary
% drafts-review PASS: D1 count-as-set sort fix, D2 dangling participle->Let, D3
% two content parens->product-zero form, D4 type-anchor tighten, D5 emph
% separable/non-separable (paper's 12-precedent definiendum convention, incl the
% twin single-side), D6 comma splice; applied + verified correct/complete; math
% re-derived ((d-2)-D identity, thm:sepcase both cases, d=4 necessity). Rubrics
% math-proofs, fit, english-usage, paper-writing, source-grounding, latex.)

A configuration is \emph{separable} if its envelope factors additively,
$E_{A,c}(i,j)=\phi(i)+\psi(j)$ for some functions $\phi,\psi$, equivalently if its
envelope has tropical rank one~\cite{DevelinSantosSturmfels2005}, and
\emph{non-separable} otherwise. Every single-side configuration is separable, since
Lemma~\ref{lem:edgered} splits its envelope into a row part and a column part. A
single-apex configuration is likewise separable, since its apex set is a $1\times1$ product
grid.

\begin{lemma}[Separability Criterion]\label{lem:sep}
An admissible configuration $(A,c)$ is separable iff its
apex set is a product grid $A=R\times C$ of a row set $R$ and a column set $C$,
and its offsets factor into a row part $\alpha_r$ and a column part
$\beta_\kappa$, $c_{(r,\kappa)}=\alpha_r+\beta_\kappa$.
\end{lemma}

\begin{proof}
Suppose first that $A=R\times C$ and $c_{(r,\kappa)}=\alpha_r+\beta_\kappa$.
Because $(r,\kappa)$ ranges over the full product $R\times C$, the row and column
minimisations are independent, so
\[
  E_{A,c}(i,j)=\min_{r\in R,\,\kappa\in C}\big(\alpha_r+\beta_\kappa+|i-r|+|j-\kappa|\big)
        =\min_{r\in R}\big(\alpha_r+|i-r|\big)+\min_{\kappa\in C}\big(\beta_\kappa+|j-\kappa|\big).
\]
The first summand depends only on $i$ and the second only on $j$, so $(A,c)$ is
separable. Conversely, suppose $E_{A,c}=\phi+\psi$. The neighbour differences
separate,
\[
  E_{A,c}(i\pm1,j)-E_{A,c}(i,j)=\phi(i\pm1)-\phi(i),\qquad
  E_{A,c}(i,j\pm1)-E_{A,c}(i,j)=\psi(j\pm1)-\psi(j).
\]
Hence $(i,j)$ is a strict local minimum of $E_{A,c}$ iff $i$ is a strict
local minimum of $\phi$ and $j$ is a strict local minimum of $\psi$; the same holds
with maxima in place of minima. Let $R$ and
$C$ be the strict local minima of $\phi$ and $\psi$. Those of $E_{A,c}$ are then exactly the product $R\times C$. Every strict local
minimum is an apex (Lemma~\ref{lem:minatapex}), and by activity every apex is a
strict local minimum (Lemma~\ref{lem:activemin}), so these minima are exactly the
apexes and $A=R\times C$. At each apex, activity makes the envelope value the offset,
$c_{(r,\kappa)}=E_{A,c}(r,\kappa)=\phi(r)+\psi(\kappa)$, exhibiting the offset
factorisation.
\end{proof}

A product-grid apex set alone does not force separability, since the offsets must
also factor, a linear condition. Separability therefore carves out a
polyhedral subfamily, splitting the count
$N_{(a,b)}=N^{\mathrm{s}}_{(a,b)}+N^{\mathrm{ns}}_{(a,b)}$ into separable and
non-separable parts, each piecewise quasi-polynomial.

\begin{lemma}[Degree identity]\label{lem:dimid}
Among the configurations counted by $N^{\mathrm{s}}_{(a,b)}$, those whose minima occupy $\rho$ rows
and $\gamma$ columns and whose maxima occupy $\rho'$ rows and $\gamma'$ columns
form a family, with $a=\rho\gamma$ and $b=\rho'\gamma'$. Its count is a product of
two one-dimensional walk counts, one per axis, nonzero only when $|\rho-\rho'|\le1$
and $|\gamma-\gamma'|\le1$. When nonzero, the count has degree $\rho+\rho'-2$ in
$m$ and $\gamma+\gamma'-2$ in $n$, so total degree
\[
  D=(\rho+\rho'-2)+(\gamma+\gamma'-2).
\]
For all $\rho,\gamma,\rho',\gamma'\ge1$,
\[
  (d-2)-D=(\rho-1)(\gamma-1)+(\rho'-1)(\gamma'-1)\ \ge\ 0.
\]
\end{lemma}

\begin{proof}
The configurations counted by $N^{\mathrm{s}}_{(a,b)}$ encode height functions, hence are
admissible, so Lemma~\ref{lem:sep} presents each
as a pair $(\phi,\psi)$ with $E_{A,c}=\phi+\psi$. Here $\phi$ is a row walk on
$\{1,\dots,m\}$ and $\psi$ a column walk on $\{1,\dots,n\}$, both $\pm1$ walks by
the unit edge differences of $E_{A,c}$ (Lemma~\ref{lem:parityvalid}). The two
walks are chosen independently, with the normalisation $E_{A,c}(1,1)=0$ fixing the
additive constant of their sum. The row walk $\phi$ carries $\rho$ minima and
$\rho'$ maxima; the column walk $\psi$ carries $\gamma$ minima and $\gamma'$ maxima.
The family count is thus the product of the two walk counts of Lemma~\ref{lem:binom}.
These vanish unless $|\rho-\rho'|\le1$ and $|\gamma-\gamma'|\le1$, and otherwise have
degrees $\rho+\rho'-2$ and $\gamma+\gamma'-2$ with nonzero leading coefficients, so the
two degrees sum to the total degree. The inequality follows from the algebraic identity
\[
  (d-2)-D=\rho\gamma+\rho'\gamma'-(\rho+\rho'+\gamma+\gamma')+2
    =(\rho-1)(\gamma-1)+(\rho'-1)(\gamma'-1),
\]
using $a=\rho\gamma$, $b=\rho'\gamma'$, $d=a+b$, and
$D=\rho+\rho'+\gamma+\gamma'-4$. Both products are non-negative, so $(d-2)-D\ge0$.
\end{proof}

Equality $D=d-2$ holds exactly when $(\rho-1)(\gamma-1)=0$ and
$(\rho'-1)(\gamma'-1)=0$, the minima and maxima each collinear. Such families
are non-empty, so the bound is attained there.

\begin{theorem}[Separable case of the degree bound]\label{thm:sepcase}
For $d\ge5$, a separable family contributes total degree $d-2$ to
$N^{\mathrm{s}}_{(a,b)}$ only if it is single-side; every other separable family
contributes total degree at most $d-3$.
\end{theorem}

\begin{proof}
By the degree identity of Lemma~\ref{lem:dimid}, $D=d-2$ holds exactly when
$(\rho-1)(\gamma-1)=0$ and $(\rho'-1)(\gamma'-1)=0$, that is when the minima and
the maxima are each collinear. The two cases below show that for $d\ge5$ the only
such family is single-side. By the $m\leftrightarrow n$ symmetry, assume $\rho=1$,
so all minima share one row.

\medskip\noindent\emph{Case $\rho'=1$.} The row walk has one strict local minimum
and one strict local maximum. Both endpoints of a $\pm1$ walk are strict local
extrema, and its extrema alternate, so a walk with only these two has them at its
endpoints and is monotone in between. Its minimum is then a boundary row. Since $\rho=1$, all $a$ minima lie on that row, so the family is single-side.

\medskip\noindent\emph{Case $\rho'\ge2$.} Then $\gamma'=1$, since $(\rho'-1)(\gamma'-1)=0$. Lemma~\ref{lem:binom}
gives the per-axis walk constraints $|\rho-\rho'|\le1$ on the row walk and
$|\gamma-\gamma'|\le1$ on the column walk. At $\rho=1$ the first gives $\rho'=2$,
and at $\gamma'=1$ the second gives $\gamma\le2$. Hence
$a=\rho\gamma=\gamma\le2$ and $b=\rho'\gamma'=2$, so $d=a+b\le4$, against
$d\ge5$.

\medskip\noindent
Only the first case survives, so $D=d-2$ forces the family to be single-side. Every
other separable family therefore has $D\ne d-2$; with $D\le d-2$ always
(Lemma~\ref{lem:dimid}) and $D$ an integer, $D\le d-3$. Its contribution to
$N^{\mathrm{s}}_{(a,b)}$ has degree $D$ when nonempty and is zero otherwise, so
total degree at most $d-3$.
\end{proof}

Conjecture~\ref{conj:G2} thus holds for the separable configurations counted by
$R_{(a,b)}$.
The restriction to $d\ge5$ is necessary. At $d=4$ the second case occurs
at $\gamma=2$, giving a separable family with two minima on an interior row, of
total degree $d-2=2$.

\subsection{The non-separable residual}
% REOPENED 2026-08-04: added the paragraph reaching d=8 via prop:converse rather
% than by enumerating families.
% Subsection 10.2 (non-separable residual) LOCKED 2026-07-12 (Adversary
% drafts-review PASS: E1' (subject-verb, retain existential 'some'), E2 ('any'
% ban), E3 (count-as-set sort fix, same pattern as 10.1 D1), E4 (false-universal
% -> 'some ... admit') applied + verified; N1/N2 left; claim-calibration honest
% (bound never claimed proven, 'certified through d=7'). ALL empirical numbers
% independently reproduced (3 enum tools): splits, deg N^ns<=d-3 (2/3/3/3/4),
% R_(3,4) deg 4, N^ns_(2,5) deg 3, d=4 (2,2)=4mn+4m+4n-36, census (3,3)->2
% (3,4)->3 (2,2)->0. Rubrics math-proofs, fit, english-usage, paper-writing,
% source-grounding, latex.)

Theorem~\ref{thm:sepcase} leaves exactly one part of Conjecture~\ref{conj:G2}
open, the non-separable bound $\deg N^{\mathrm{ns}}_{(a,b)}\le d-3$ for $d\ge5$.
For $a=2$, the cost of non-separability is visible locally. A non-separable pair
shares neither a row nor a column, so its two cones meet along a diagonal ridge.
Requiring exactly $b$ maxima pins the ridge length, coupling the two axes and
removing one degree of freedom. This recovers the Ridge Lemma of~\cite{Gupta2026companion} as
the two-apex case of the Maxima Criterion (Lemma~\ref{lem:maxima}).

For $a\ge3$, this localisation fails. A non-separable apex set still contains a diagonal
pair, but the pair need not meet along an exposed ridge. For instance, some
non-separable configurations with apexes on two parallel rows at four distinct
columns have cones meeting only along axis-parallel bisectors, with no diagonal
ridge. The degree drop is then a
global property of the configuration, the joint effect of the maxima constraint
across all $a$ cones, rather than the signature of a single local feature. This
is what a uniform argument must capture, and what makes the ridge-based method that
explains the $a=2$ drop insufficient on its own.
The envelope of Section~\ref{sec:envelope} is a weighted Voronoi diagram, and its
combinatorial types stratify the space of configurations in the tropical
framework, yet the bound is open even there.

Absent such a uniform argument, the residual is certified through $d=7$ by direct
enumeration. This enumeration reads the exact total degree of each family by finite
differences. The reading is faithful because each family contributes a nonnegative
count, so nonnegative top-degree parts cannot cancel and no family of degree $d-2$
can hide. No non-separable family
reaches $d-2$ through $d=7$, across the splits $(2,3)$ for $d=5$, $(2,4)$ and
$(3,3)$ for $d=6$, and $(2,5)$ and $(3,4)$ for $d=7$. At $d=7$, for instance,
the top-degree-$5$ part of $N_{(3,4)}$ is entirely the single-side contribution
(Proposition~\ref{prop:singleside}), so its remainder $R_{(3,4)}$ carries no degree-$5$
monomial. No separable family meets $a=2$, $b=5$ under the walk constraints of
Lemma~\ref{lem:dimid}, so every configuration counted by $N_{(2,5)}$ is
non-separable, and $N_{(2,5)}$ carries no degree-$5$ monomial either. Both $\deg R_{(3,4)}$ and $\deg N^{\mathrm{ns}}_{(2,5)}$ are therefore at most
$4=d-3$. The separable formula
$\rho+\gamma+\rho'+\gamma'-4$ is exact for separable families
(Lemma~\ref{lem:dimid}). It is observed to bound the degree of the non-separable
families as well, with strict inequality in every computed case. But it exceeds
$d-3$ for families with widely spread maxima, so it does not close the gap on
its own.

At $d=8$ the residual is reached from the other end, without inspecting the
families. The interpolation of Section~\ref{sec:data} determines $p_8$
without assuming the bound, and its top-degree part is exactly the single-side
contribution; every remainder must then fall short of degree $d-2$
(Proposition~\ref{prop:converse}, Corollary~\ref{cor:G2d8}). This route needs
only the top-degree part of a polynomial the construction already produces, so it
extends to any degree at which $p_d$ can be determined independently of the
bound. What it does not give is a uniform argument, since it must be rerun degree
by degree.

The residual statement is again special to $d\ge5$. At $d=4$, the
balanced split $(2,2)$ has a non-separable family of two apexes on
the two parallel horizontal sides at distinct columns. Its independent cross-axis
offsets supply an $mn$ top term, so it reaches total degree $d-2=2$. A census of apex
positions shows that some configurations with $a\ge2$ admit interior apexes, up
to $2$ for the split $(3,3)$ and up to $3$ for $(3,4)$. The residual is therefore
genuinely a degree statement rather than the claim that all apexes lie on the
boundary, which holds only at $d=4$.

% ===========================================================================
\section{Explicit polynomials and conjectures}\label{sec:data}
% REOPENED 2026-08-04: added cor:G2d8 applying prop:converse to the
% unconditionally determined p_8.
% Section 11 (sec:data) LOCKED 2026-07-12 (14 hardened prose edits: 4 prose
% colons C1-C4, 9 content parentheticals P1-P8+M1, 1 cross-section
% single-side-term reconciliation X1; two Adversary hardening passes +
% post-apply Adversary verification, all 14 clean, zero residual colons/parens).
% ===========================================================================

\subsection{Closed forms on the high region}\label{subsec:high}

Assembling $E_d$ from the envelope enumeration and interpolating on the high
region (Theorem~\ref{thm:poly}) gives the following closed forms. The two lowest
cases, $p_2\equiv4$ and
$p_3=4(m+n)-16$, are elementary, both of total degree $d-2$. The next two,
$p_4$ and $p_5$, recover the formulas of~\cite{Gupta2026companion}; $p_6$ and
$p_7$ are new; $p_8$, $p_9$ and $p_{10}$ extend the list via the transfer matrix
of Lemma~\ref{lem:ratGF}. The envelope-enumeration forms $p_2$--$p_7$ were
cross-checked against two independent computations, direct enumeration of
$\OFG(\Mm)$ on small grids and the transfer matrix.

The samples come from the transfer matrix in every case, but two interpolations
fit them under different assumptions.
Lemma~\ref{lem:uniform} bounds the degree in each variable by $d-2$
unconditionally, so the tensor grid whose rows and columns run over
$d-1,\dots,2d-3$ determines $p_d$ with no further hypothesis, leaving the total
degree to be read off rather than imposed. That route gives $p_8$, from the rows
$m=7,\dots,13$. The fitted polynomial carries no monomial of total degree above
$d-2$, so the total degree $d-2$ is observed and not assumed, and it agrees with $E_8$ at
fourteen further points off the grid. Fitting a symmetric form of total degree
$d-2$ instead needs fewer nodes, on rows no higher than
$\lfloor(3d-4)/2\rfloor$, but presupposes Conjecture~\ref{conj:G2}. That cheaper
route gives $p_9$ and
$p_{10}$, whose agreement at held-out nodes is accordingly a test of the
conjecture at those degrees rather than an independent check of the polynomial.
The forms $p_9$ and $p_{10}$ below are thus conditional on
Conjecture~\ref{conj:G2}, while $p_2,\dots,p_8$ are unconditional.
{\allowdisplaybreaks
\begin{align*}
  p_4 &= 2(m^2+n^2) + 6mn - 10(m+n) - 4,\\
  p_5 &= \tfrac{2}{3}(m^3+n^3) + 2(m^2+n^2) + 50mn - \tfrac{392}{3}(m+n) + 264,\\
  p_6 &= \tfrac{1}{6}(m^4+n^4) + \tfrac{1}{3}(m^3+n^3) + 38(m^2n+mn^2)\\
      &\quad{} - \tfrac{229}{6}(m^2+n^2) - 272mn - \tfrac{103}{3}(m+n) + 1176,\\
  p_7 &= \tfrac{1}{30}(m^5+n^5) + \tfrac{37}{3}(m^3n+mn^3) + 18m^2n^2
         + \tfrac{7}{6}(m^3+n^3) + 25(m^2n+mn^2)\\
      &\quad{} - 496(m^2+n^2) - \tfrac{3818}{3}mn + \tfrac{14354}{5}(m+n) - 904,\\
  p_8 &= \tfrac{1}{180}(m^6+n^6)\\
      &\quad{} - \tfrac{1}{60}(m^5+n^5) + 3(m^4n+mn^4) + \tfrac{14}{3}(m^3n^2+m^2n^3)\\
      &\quad{} + \tfrac{95}{36}(m^4+n^4) + \tfrac{91}{3}(m^3n+mn^3) + 210 m^2n^2\\
      &\quad{} - \tfrac{1399}{12}(m^3+n^3) - \tfrac{2780}{3}(m^2n+mn^2)\\
      &\quad{} - \tfrac{41609}{45}(m^2+n^2) - \tfrac{11822}{3}mn
        + \tfrac{125353}{5}(m+n) - 56788,\\
  p_9 &= \tfrac{1}{1260}(m^7+n^7)\\
      &\quad{} - \tfrac{1}{180}(m^6+n^6) + \tfrac{11}{15}(m^5n+mn^5)
        + (m^4n^2+m^2n^4) + 2m^3n^3\\
      &\quad{} + \tfrac{97}{180}(m^5+n^5) + \tfrac{7}{2}(m^4n+mn^4)
        + 99(m^3n^2+m^2n^3)\\
      &\quad{} - \tfrac{155}{36}(m^4+n^4) + \tfrac{449}{3}(m^3n+mn^3) - 400 m^2n^2\\
      &\quad{} - \tfrac{68852}{45}(m^3+n^3) - \tfrac{16657}{2}(m^2n+mn^2)\\
      &\quad{} + \tfrac{640274}{45}(m^2+n^2) + \tfrac{220166}{5}mn
        + \tfrac{4101458}{105}(m+n) - 368848,\\
  p_{10} &= \tfrac{1}{10080}(m^8+n^8)\\
         &\quad{} - \tfrac{1}{840}(m^7+n^7) + \tfrac{13}{90}(m^6n+mn^6)
           + \tfrac{4}{15}(m^5n^2+m^2n^5) + \tfrac{7}{18}(m^4n^3+m^3n^4)\\
         &\quad{} + \tfrac{67}{720}(m^6+n^6) - \tfrac{21}{20}(m^5n+mn^5)
           + \tfrac{91}{3}(m^4n^2+m^2n^4) + \tfrac{472}{9}m^3n^3\\
         &\quad{} + \tfrac{101}{30}(m^5+n^5) + \tfrac{2101}{18}(m^4n+mn^4)
           + \tfrac{3407}{18}(m^3n^2+m^2n^3)\\
         &\quad{} - \tfrac{421393}{1440}(m^4+n^4) - \tfrac{120535}{36}(m^3n+mn^3)
           - \tfrac{10004}{3}m^2n^2\\
         &\quad{} - \tfrac{562697}{120}(m^3+n^3) - \tfrac{159259}{5}(m^2n+mn^2)\\
         &\quad{} + \tfrac{455062963}{2520}(m^2+n^2) + \tfrac{19134272}{45}mn\\
         &\quad{} - \tfrac{137300147}{210}(m+n) - 885004.
\end{align*}}%
\noindent The top-degree coefficient matches $C(d)=4/(d-2)!$ of
Proposition~\ref{prop:singleside} for $d\ge3$. Only $d=4$ carries a mixed
top monomial (Table~\ref{tab:topdeg}).

At $d=8$ this reading is unconditional, and that is enough to settle the degree
bound there.

\begin{corollary}[The degree bound at $d=8$]\label{cor:G2d8}
Every remainder $R_{(a,b)}$ with $a+b=8$ and $a\le b$ has total degree at most
$5$. Equivalently, Conjecture~\ref{conj:G2} holds at $d=8$, the non-separable
splits included.
\end{corollary}

\begin{proof}
The tensor grid $m,n=7,\dots,13$ determines $p_8$ through
Lemma~\ref{lem:uniform} alone, with no appeal to Conjecture~\ref{conj:G2}. The
resulting polynomial carries no monomial of total degree above $6=d-2$, and its
total-degree-$6$ part is $\tfrac{1}{180}(m^6+n^6)$, which is $C(8)\,(m^6+n^6)$.
Proposition~\ref{prop:converse} applies.
\end{proof}

\begin{table}[!ht]
\centering
\begin{tabular}{@{}cccc@{}}
\toprule
$d$ & $C(d)=4/(d-2)!$ & Non-mixed top-degree part & Mixed top-degree part \\
\midrule
$2$ & ---              & $p_2\equiv 4$                & ---   \\
$3$ & $4$              & $4(m+n)$                    & ---   \\
$4$ & $2$              & $2(m^2+n^2)$                & $6mn$ \\
$5$ & $\tfrac{2}{3}$   & $\tfrac{2}{3}(m^3+n^3)$     & ---   \\
$6$ & $\tfrac{1}{6}$   & $\tfrac{1}{6}(m^4+n^4)$     & ---   \\
$7$ & $\tfrac{1}{30}$  & $\tfrac{1}{30}(m^5+n^5)$    & ---   \\
$8$  & $\tfrac{1}{180}$   & $\tfrac{1}{180}(m^6+n^6)$   & ---   \\
$9$  & $\tfrac{1}{1260}$  & $\tfrac{1}{1260}(m^7+n^7)$  & ---   \\
$10$ & $\tfrac{1}{10080}$ & $\tfrac{1}{10080}(m^8+n^8)$ & ---   \\
\bottomrule
\end{tabular}
\caption{Top-degree parts of $p_d$. The constant $p_2$ lies outside the
top-degree law, shown dashed. The rows $d=9,10$ presuppose
Conjecture~\ref{conj:G2} (Section~\ref{subsec:high}).}
\label{tab:topdeg}
\end{table}

\subsection{The sub-top-degree coefficient}\label{subsec:subtop}

The degree bound of Conjecture~\ref{conj:G2} is the one open assumption the main
results rest on. The further conjectures below refine the sub-leading and boundary
structure without bearing on $p_d$ or its degree. These are Conjectures~\ref{conj:subtop}
and~\ref{conj:c2} here, and Conjecture~\ref{conj:baxter} in the next subsection.

The sub-top pure coefficient obeys a per-split law.
Lemma~\ref{lem:splitedge} splits each $N_{(a,b)}$ into its single-side term and
an interior remainder $R_{(a,b)}(m,n) = N_{(a,b)}(m,n) - 2f_{(a,b)}(n) - 2f_{(a,b)}(m)$,
which carries that coefficient.

\begin{conjecture}[Sub-top pure coefficient, per split]\label{conj:subtop}
For every $d\ge5$ and every split $(a,b)$ with $a+b=d$ and $a\le b$, the
coefficient of $m^{d-3}$, equivalently $n^{d-3}$, in $R_{(a,b)}$
is
\[
  [m^{d-3}]\,R_{(a,b)}(m,n) \;=\;
  \begin{cases}
    \dfrac{8\,(1+\delta_{ab})}{(d-3)!} & \text{if } |a-b|\le 1, \\[2pt]
    0 & \text{if } |a-b|>1.
  \end{cases}
\]
\end{conjecture}

\begin{corollary}[Sub-top pure of $E_d$]\label{cor:subtop}
Under Conjecture~\ref{conj:subtop}, for every $d\ge5$,
\[
  [m^{d-3}]\,E_d(m,n) \;=\; \frac{2(7-d)}{(d-3)!}.
\]
\end{corollary}

\begin{proof}
By~\eqref{eq:Ed}, $E_d = \sum_{(a,b):\,a\le b}
(2-\delta_{ab})\,N_{(a,b)}$; expanding each $N_{(a,b)}$ by
Lemma~\ref{lem:splitedge} and extracting $[m^{d-3}]$,
\[
  [m^{d-3}]\,E_d = \sum_{(a,b)} (2-\delta_{ab})\bigl(
     [m^{d-3}]\,2f_{(a,b)}(m) + [m^{d-3}]\,R_{(a,b)}
  \bigr),
\]
since $f_{(a,b)}(n)$ contributes nothing to $[m^{d-3}]$. By Lemma~\ref{lem:binom}
the single-side count $2f_{(a,b)}(m)$ vanishes unless $|a-b|\le1$, so among the
$a\le b$ splits only the near-balanced one contributes; there
$2f_{(a,b)}(m)=2(1+\delta_{ab})\binom{m-2}{d-2}$. Expanding the falling factorial,
\[
  [m^{d-3}]\binom{m-2}{d-2}=-\frac{1}{(d-2)!}\sum_{k=2}^{d-1}k
    =-\frac{(d-2)(d+1)}{2\,(d-2)!}=-\frac{d+1}{2(d-3)!},
\]
so $[m^{d-3}]\,2f_{(a,b)}(m)=-(1+\delta_{ab})(d+1)/(d-3)!$; weighting by
$(2-\delta_{ab})$ gives $-2(d+1)/(d-3)!$.
The $R$ term, by Conjecture~\ref{conj:subtop}, contributes
$(2-\delta_{ab})(1+\delta_{ab})\cdot 8/(d-3)! = 16/(d-3)!$. Summing,
$16/(d-3)! - 2(d+1)/(d-3)! = (14-2d)/(d-3)! = 2(7-d)/(d-3)!$.
\end{proof}

Combining Corollary~\ref{cor:subtop} with Lemma~\ref{lem:splitedge} and
Lemma~\ref{lem:binom} gives, for every $d\ge5$,
\[
  p_d(m,n) = 4\binom{m-2}{d-2} + 4\binom{n-2}{d-2}
           + \frac{16}{(d-3)!}\bigl(m^{d-3}+n^{d-3}\bigr)
           + \tilde R_d(m,n),
\]
where the residual $\tilde R_d$ has total degree at most $d-3$ under
Conjecture~\ref{conj:G2}, and unconditionally for $d\le8$. By
Corollary~\ref{cor:subtop} it carries no pure $m^{d-3}$ or $n^{d-3}$ monomial. The
single-side term is proved; the sub-top pure layer follows from
Corollary~\ref{cor:subtop}; the residual $\tilde R_d$ is uncharacterised.

\begin{remark}[Explicit residuals at small $d$]
For the first few cases:
{\allowdisplaybreaks
\begin{align*}
  \tilde R_5(m,n) &= 50\,mn - 148(m+n) + 296,\\
  \tilde R_6(m,n) &= 38(m^2n+mn^2) - 50(m^2+n^2) - 272\,mn\\
                   &\quad{} - \tfrac{26}{3}(m+n) + 1136,\\
  \tilde R_7(m,n) &= \tfrac{37}{3}(m^3n+mn^3) + 18\,m^2n^2
                      + 25(m^2n+mn^2) - 4(m^3+n^3)\\
                   &\quad{} - \tfrac{1430}{3}(m^2+n^2) - \tfrac{3818}{3}mn
                      + 2836(m+n) - 856.
\end{align*}}%
Each has total degree exactly $d-3$. The mixed top layer is nonzero in every
case, namely $50\,mn$, $38(m^2n+mn^2)$, and $\tfrac{37}{3}(m^3n+mn^3) + 18\,m^2n^2$. The
pure-power constraint of Conjecture~\ref{conj:subtop} does not reach these mixed terms.
\end{remark}

Numerically, the rescaled coefficient $(d-3)!\,[m^{d-3}]E_d(m,n) = 2(7-d)$ is
linear in $d$, verified at $d=5,\dots,10$ by the transfer matrix of
Lemma~\ref{lem:ratGF}, and for $d\le7$ independently by exact-integer finite
differences on samples of $p_d$. The vanishing at non-near-balanced splits underlying
Conjecture~\ref{conj:subtop} was verified at $d=5,6,7,8$ across all $(a,b)$
splits by the transfer matrix extended to track $(\#\mathrm{min},
\#\mathrm{max})$ separately. The type is detected by the colour-rotation-invariant
rule $c\equiv c'+1\pmod3$ for a minimum versus $c\equiv c'-1\pmod3$ for a maximum, where $c$ is
the shared neighbour colour and $c'$ is the cell colour.

The transfer-matrix setting offers a second view of the sub-top coefficient. The generating
function of $E_d(m,n)$ in $n$ is rational with its only pole at $z=1$
(Lemma~\ref{lem:quotient}), of order at most $d-1$ (Lemma~\ref{lem:uniform}). Its Laurent expansion is
\[
  \sum_{n\ge2}E_d(m,n)\,z^n
   = \frac{c_1(m)}{(1-z)^{d-1}} + \frac{c_2(m)}{(1-z)^{d-2}} + \cdots,
\]
and Proposition~\ref{prop:leading} pins $c_1(m)=4$ uniformly in $m\ge2$.

\begin{conjecture}[Sub-leading Laurent coefficient]\label{conj:c2}
For every $d\ge5$ and every $m\ge2$,
\[
  c_2(m) \;=\; -4(d-4).
\]
\end{conjecture}

This was verified at $d=5,6,7,8$ across $m\in\{d-1,d,d+1,d+2,d+3\}$, twenty
cases in all, by exact-rational polynomial fit of $E_d(m,n)$.
Conjecture~\ref{conj:c2} splits into two components, that $c_2(m)$ is
$m$-independent and that its value is $-4(d-4)$.

\begin{proposition}[Laurent coefficient and the degree bound]\label{prop:laurentrel}
Under Conjecture~\ref{conj:G2}, the coefficient $c_2(m)$ is independent of $m$.
Moreover, Conjecture~\ref{conj:c2} implies Corollary~\ref{cor:subtop},
independently of Conjecture~\ref{conj:subtop}.
\end{proposition}

\begin{proof}
The Laurent expansion gives
\[
  [n^{d-3}]E_d(m,n) = \frac{(d-1)\,c_1(m)}{2\,(d-3)!} + \frac{c_2(m)}{(d-3)!},
\]
and $c_1(m)=4$ by Proposition~\ref{prop:leading}, so
$[m^i\,n^{d-3}]E_d = [m^i]\,c_2(m)/(d-3)!$ for every $i\ge1$.

\medskip\noindent
Under Conjecture~\ref{conj:G2}, the top-degree part of $p_d$ is the
single-side layer $C(d)(m^{d-2}+n^{d-2})$ by Proposition~\ref{prop:leading}, so
$[m^i\,n^{d-3}]E_d = 0$ for every $i\ge1$. At $i=1$ the monomial $m\,n^{d-3}$
has total degree $d-2$ but is mixed, while that layer carries only pure powers;
for $i\ge2$ it exceeds degree $d-2$. Hence $[m^i]\,c_2(m)=0$ for
every $i\ge1$, so $c_2(m)$ is constant in $m$.

\medskip\noindent
The $m$-free part of the expansion is
$[m^0\,n^{d-3}]E_d = 2(d-1)/(d-3)! + [m^0]\,c_2(m)/(d-3)!$; with
$c_2(m)=-4(d-4)$ from Conjecture~\ref{conj:c2} this equals $2(7-d)/(d-3)!$.
The symmetry $p_d(m,n)=p_d(n,m)$ then gives
$[m^{d-3}]E_d = [m^0\,n^{d-3}]E_d = 2(7-d)/(d-3)!$, which is
Corollary~\ref{cor:subtop}.
\end{proof}

\subsection{Boundary corrections below threshold}\label{subsec:below}

Below threshold, with $s=\min(m,n)<d-1$ the smaller dimension and $n$ the
larger, the count differs from $p_d$ by the correction
\[
  B_d^{(s)}(n) := E_d(s,n) - p_d(s,n).
\]
Lemma~\ref{lem:uniform}, applied at $m=s\ge2$, makes $E_d(s,\cdot)$ agree on
$\{n\ge d-1\}$ with a single polynomial in $n$, so $B_d^{(s)}(\cdot)$ is a
polynomial in $n$ on that range, as Table~\ref{tab:corr} records.
\begin{table}[!ht]
\centering
\begin{tabular}{@{}llll@{}}
\toprule
$d$ & $s$ & $B_d^{(s)}(n)$ & $\deg$\\
\midrule
$4$ & $2$ & $12-4n$ & $1$\\
$5$ & $2$ & $40$ & $0$\\
$5$ & $3$ & $32-8n$ & $1$\\
$6$ & $2$ & $\tfrac43n^3-62n^2+\tfrac{1538}{3}n-1056$ & $3$\\
$6$ & $3$ & $-6n^2+30n+80$ & $2$\\
$6$ & $4$ & $120-24n$ & $1$\\
$7$ & $2$ & $\tfrac23n^4-\tfrac{106}{3}n^3+\tfrac{1204}{3}n^2
            -\tfrac{1484}{3}n-3016$ & $4$\\
$7$ & $3$ & $6n^3-244n^2+2390n-6024$ & $3$\\
$7$ & $4$ & $-28n^2+196n+104$ & $2$\\
$7$ & $5$ & $528-88n$ & $1$\\
\bottomrule
\end{tabular}
\caption{Boundary corrections below threshold.}
\label{tab:corr}
\end{table}
Each correction has degree $d-1-s$, linear at $s=d-2$ and rising as $s$
decreases. The lone departure is $B_5^{(2)}$, a constant where the pattern gives
degree $2$. Every other pair computed follows the pattern, including each of
$s=2,\dots,6$ at $d=8$.

\begin{conjecture}[Baxter onset defect]\label{conj:baxter}
For every $d\ge4$, the correction $B_d^{(d-2)}(n)$ is linear in $n$
with leading coefficient $-4\,\mathrm{Bax}(d-3)$, where $\mathrm{Bax}(k)$ is the
$k$-th Baxter number $1,2,6,22,92,\dots$\footnote{\url{https://oeis.org/A001181}; the
listing there repeats the initial term, so $\mathrm{Bax}(k)$ is its $(k{+}1)$-th entry.}.
\end{conjecture}

The leading coefficients $-4,-8,-24,-88$ in Table~\ref{tab:corr} are $-4$ times
the Baxter numbers $1,2,6,22$. The transfer matrix of Lemma~\ref{lem:ratGF}
extends the agreement through $d=11$, matching the next four Baxter numbers
$92,422,2074,10754$. Since the Baxter numbers are positive,
Conjecture~\ref{conj:baxter} would make every $B_d^{(d-2)}(\cdot)$ nonzero. The count
would then depart from $p_d$ at the smaller dimension $s=d-2$, one step below
threshold, so the threshold $d-1$ of the closed form would be sharp for every $d$. No proof is known; Section~\ref{sec:discussion} makes the obstruction explicit through
a finite-support form of the underlying transfer object.

\begin{proposition}[Boundary-correction structure]\label{prop:meta}
For every $d\ge4$ and every $m,n$ with exactly one of them below $d-1$, the
count $E_d(m,n)$ differs from $p_d(m,n)$ by a polynomial in the larger
variable, determined by $\min(m,n)$.
\end{proposition}

\begin{proof}
By the symmetry of $E_d$ and $p_d$, let $s=\min(m,n)$, so that $s<d-1$ and the
larger dimension is at least $d-1$. For $s\ge2$, Lemma~\ref{lem:uniform} makes
$E_d(s,\cdot)$ agree on $\{n\ge d-1\}$ with a single polynomial in $n$, and
$p_d(s,\cdot)$ is a polynomial in $n$, so the difference $B_d^{(s)}$ is one too,
and it depends on the pair only through $s$.

\medskip\noindent
For $s=1$ the grid is the path $P_n$, on which a height function is a $\pm1$
walk and the extrema are its strict local extrema, endpoints included. By
Lemma~\ref{lem:binom} the number of such walks with $a$ minima and $b$ maxima
is $(1+\delta_{ab})\binom{n-2}{d-2}$ when $|a-b|\le1$ and $0$ otherwise. Among
the splits of $d$ the surviving ones are the single balanced split for $d$ even
and the two near-balanced splits for $d$ odd, so either way the sum is
$E_d(1,n)=2\binom{n-2}{d-2}$, a polynomial in $n$ determined by $s$.
\end{proof}

The finitely many pairs with both dimensions below $d-1$, which arise only for
$d\ge4$, are covered by direct computation.

% ===========================================================================
\section{Discussion}\label{sec:discussion}
% Section 12 (sec:discussion) LOCKED 2026-07-12 (R1 display-punctuation fix
% "supported"->"It is supported" + R2 pronoun-naming "realises it"->"realises
% that subfamily"; Adversary-hardened drafts, all six math checks pass, zero
% residual prose colons/parentheticals across all four paragraphs).
% ===========================================================================

For every degree, the construction reduces the vertex count to a lattice-point
enumeration. A height function is encoded by an integer configuration
(Theorem~\ref{thm:envelope}), and the configurations with exactly $d$ strict
local extrema satisfy the linear Maxima-Criterion inequalities
(Lemma~\ref{lem:maxima}). Their number is therefore piecewise quasi-polynomial
in $(m,n)$ (Theorem~\ref{thm:quasipoly}), and on the high region
$m,n\ge\max(d-1,2)$ a single symmetric polynomial $p_d$
(Theorem~\ref{thm:poly}), explicit through $d=10$. The existence, symmetry,
region, and per-axis degree of $p_d$ are unconditional. Two features rest on
conjecture, and the construction turns each from a question about the flip graph
into one about the lattice-point count.

The first, whether the total degree stays $d-2$, is the degree bound
(Conjecture~\ref{conj:G2}). For each split the admissible configurations form a
Presburger family (Theorem~\ref{thm:quasipoly}). Fixing which
Maxima-Criterion inequalities are tight cuts it into subfamilies, each counted
by a quasi-polynomial whose total degree counts free axis parameters.
Lemma~\ref{lem:dimid} makes this explicit for the separable families. Total degree
$d-2$ then requires a subfamily with $d-2$ such parameters, which the bound
asserts no non-separable configuration supplies. Whether a subfamily reaches
that degree is fixed by which inequalities are tight, independently of whether a
height function realises it. The two directions are not symmetric. Excluding
every non-separable subfamily with $d-2$ free parameters would prove the bound,
whereas exhibiting one would not refute it unless a height function realises that subfamily.
The construction settles the bound where the count factors into independent row
and column contributions (Theorem~\ref{thm:sepcase}); classifying the
non-separable subfamilies is what remains open.

The second, where the count departs from $p_d$, is the boundary. Below the
threshold $m,n\ge\max(d-1,2)$, each count differs from $p_d$ by a correction
whose leading coefficients Conjecture~\ref{conj:baxter} identifies with the
Baxter numbers. These enumerate Baxter permutations and non-intersecting
lattice-path triples\footnote{\url{https://oeis.org/A001181}} and correspond to
plane bipolar orientations and quadrant walks~\cite{Felsner2011,BousquetMelou2019}. The obstruction to a proof
is the dependence on the number of rows. At fixed $m$, the correction has a
rational generating function from a finite transfer matrix, whereas the Baxter
numbers are $P$-recursive and surface only as $m$ grows. The boundary quantity
$\xi_k(m):=\mathbf u_1^{\!\top}(B_1Q)^kB_1\mathbf v_1$ is built from the
single-extremum boundary vectors $\mathbf u_1,\mathbf v_1$ and the column marker
$B_1$ of Section~\ref{sec:period}. Through $m=9$ it has the closed form
\[
  \xi_k(m)=
  \begin{cases}
    12 & k=0,\\
    12(m+2-k)\,\mathrm{Bax}(k) & 1\le k\le m-2,\\
    24\,\mathrm{Bax}(m-1) & k=m-1,\\
    0 & k\ge m.
  \end{cases}
\]
This form is verified at those widths rather than proved. It is supported on
$0\le k\le m-1$ at fixed $m$, so a single width shows only
$\mathrm{Bax}(1),\dots,\mathrm{Bax}(m-1)$ and no width carries the whole sequence;
the Baxter numbers emerge only across widths, in the difference
$\xi_k(m+1)-\xi_k(m)=12\,\mathrm{Bax}(k)$ for $1\le k\le m-2$. The rotation
quotient (Lemma~\ref{lem:quotient}) divides this increment by three to
$4\,\mathrm{Bax}(k)$, the magnitude at $k=d-3$ of the boundary correction's
leading coefficient $-4\,\mathrm{Bax}(d-3)$ (Conjecture~\ref{conj:baxter}). Before the quotient the
transfer object already exhibits this Baxter dependence, matching through $d=11$. A proof must
therefore control the passage from $m$ to $m+1$ rows, through a sign-reversing
involution realising the corrections as non-intersecting lattice paths, or a
functional equation for the row-refined generating function.

The two obstructions are internal to the enumeration, yet the degree sequence
itself carries a direct mechanical reading. The locally flat-foldable
mountain-valley
assignments of the Miura-ori are the candidate states of a reconfigurable
metamaterial~\cite{Zhai2021metamaterials}, and a face flip is a single
reconfiguration between them~\cite{Pratapa2021reprogrammable}. Mountain-valley
reassignment underlies the stable states of a Miura
stack~\cite{Liu2021miura} and the multistable, reprogrammable devices built from
it~\cite{Jamalimehr2022lockable,Liu2023cellular}. A vertex degree is then the
number of single-flip reconfigurations available to a state. For $d\ge5$, and
subject to the degree bound, the number of states with
$d$ reconfigurations grows like $m^{d-2}+n^{d-2}$. The single-side configurations
carry this top-degree part, and their apexes lie along one boundary line. The degree sequence thus records how
many locally flat-foldable states offer each number of single-flip
reconfigurations; which of
them a device realises is set by global foldability and by its energy landscape,
both outside the model.

% ===========================================================================
\section{Conclusion}\label{sec:conclusion}
% ===========================================================================

The Miura-ori flip-graph degree sequence was known only in special cases, each
by a separate argument. One construction computes it uniformly, giving, for every $d$,
the number of degree-$d$ vertices as a single symmetric
polynomial $p_d(m,n)$ on the region $m,n\ge\max(d-1,2)$. These polynomials are
explicit through $d=10$, with total degree $d-2$ subject to a degree bound. That
bound is proved whenever the count splits into independent row and column
factors, and it holds in every case through $d=8$. Two questions remain
open, whether the degree bound holds
beyond that split, and where the count first departs from $p_d$ below the region,
the boundary at which the Baxter numbers surface. The degree sequence is an
invariant of the Miura-ori design space, and one construction now reaches it at
every degree.

\section*{Acknowledgements}
Simon Watts, an independent researcher, is thanked for a valuable correspondence,
and for independently proving the degree bound at $d=8$
(Corollary~\ref{cor:G2d8})~\cite{Watts2026}.

\bibliographystyle{alpha}
\bibliography{references}

\newcommand{\etalchar}[1]{$^{#1}$}
\begin{thebibliography}{HMNTS22}

\bibitem[ADE{\etalchar{+}}20]{Akitaya2020faceflips}
Hugo~A. Akitaya, Vida Dujmovi\'c, David Eppstein, Thomas~C. Hull, Kshitij Jain,
  and Anna Lubiw.
\newblock Face flips in origami tessellations.
\newblock {\em Journal of Computational Geometry}, 11(1), 2020.
\newblock Preliminary version: arXiv:1910.05667.

\bibitem[Aur91]{Aurenhammer1991}
Franz Aurenhammer.
\newblock Voronoi diagrams: A survey of a fundamental geometric data structure.
\newblock {\em ACM Computing Surveys}, 23(3):345--405, 1991.

\bibitem[BMFR20]{BousquetMelou2019}
Mireille Bousquet-M{\'e}lou, {\'E}ric Fusy, and Kilian Raschel.
\newblock Plane bipolar orientations and quadrant walks.
\newblock {\em S\'eminaire Lotharingien de Combinatoire}, 81:Article B81l,
  2020.

\bibitem[BW03]{BarvinokWoods2003}
Alexander Barvinok and Kevin Woods.
\newblock Short rational generating functions for lattice point problems.
\newblock {\em Journal of the American Mathematical Society}, 16(4):957--979,
  2003.

\bibitem[BW22]{BogartWoods2022}
Tristram Bogart and Kevin Woods.
\newblock A plethora of polynomials: A toolbox for counting problems.
\newblock {\em The American Mathematical Monthly}, 129(3):203--222, 2022.

\bibitem[CHO{\etalchar{+}}25]{Christensen2025origami}
Lumi Christensen, Thomas~C. Hull, Emma O'Neil, Valentina Pappano, Natalya
  Ter-Saakov, and Kacey Yang.
\newblock The origami flip graph of the $2 \times n$ {Miura}-ori, 2025.

\bibitem[CJS22]{CriadoJoswigSantos2021}
Francisco Criado, Michael Joswig, and Francisco Santos.
\newblock Tropical bisectors and voronoi diagrams.
\newblock {\em Foundations of Computational Mathematics}, 22(6):1923--1960,
  2022.

\bibitem[CvdHJ09]{Cereceda2009mixing}
Luis Cereceda, Jan van~den Heuvel, and Matthew Johnson.
\newblock Mixing 3-colourings in bipartite graphs.
\newblock {\em European Journal of Combinatorics}, 30(7):1593--1606, 2009.

\bibitem[DS04]{DevelinSturmfels2004}
Mike Develin and Bernd Sturmfels.
\newblock Tropical convexity.
\newblock {\em Documenta Mathematica}, 9:1--27, 2004.

\bibitem[DSS05]{DevelinSantosSturmfels2005}
Mike Develin, Francisco Santos, and Bernd Sturmfels.
\newblock On the rank of a tropical matrix.
\newblock In Jacob~E. Goodman, J{\'a}nos Pach, and Emo Welzl, editors, {\em
  Combinatorial and Computational Geometry}, volume~52 of {\em MSRI
  Publications}, pages 213--242. Cambridge University Press, 2005.

\bibitem[FFNO11]{Felsner2011}
Stefan Felsner, {\'E}ric Fusy, Marc Noy, and David Orden.
\newblock Bijections for {Baxter} families and related objects.
\newblock {\em Journal of Combinatorial Theory, Series A}, 118(3):993--1020,
  2011.

\bibitem[GH14]{GineproHull2014counting}
Jessica Ginepro and Thomas~C. Hull.
\newblock Counting {Miura}-ori foldings.
\newblock {\em Journal of Integer Sequences}, 17(10):Article 14.10.8, 2014.

\bibitem[Gup26]{Gupta2026companion}
Chakshu Gupta.
\newblock Height functions on the $m \times n$ {Miura}-ori flip graph: degree
  sequence and diameter, 2026.

\bibitem[HMNTS22]{Hull2022maximal}
Thomas~C. Hull, Manuel Morales, Sarah Nash, and Natalya Ter-Saakov.
\newblock Maximal origami flip graphs of flat-foldable vertices: Properties and
  algorithms.
\newblock {\em Journal of Graph Algorithms and Applications}, 26(4):503--517,
  2022.

\bibitem[Hul03]{Hull2003counting}
Thomas~C. Hull.
\newblock Counting mountain-valley assignments for flat folds.
\newblock {\em Ars Combinatoria}, 67:175--187, 2003.

\bibitem[JMAP22]{Jamalimehr2022lockable}
Amin Jamalimehr, Morad Mirzajanzadeh, Abdolhamid Akbarzadeh, and Damiano
  Pasini.
\newblock Rigidly flat-foldable class of lockable origami-inspired
  metamaterials with topological stiff states.
\newblock {\em Nature Communications}, 13:1816, 2022.

\bibitem[LFXW21]{Liu2021miura}
Zuolin Liu, Hongbin Fang, Jian Xu, and K.~W. Wang.
\newblock A novel origami mechanical metamaterial based on {Miura}-variant
  designs: exceptional multistability and shape reconfigurability.
\newblock {\em Smart Materials and Structures}, 30(8):085029, 2021.

\bibitem[LFXW23]{Liu2023cellular}
Zuolin Liu, Hongbin Fang, Jian Xu, and K.~W. Wang.
\newblock Cellular automata inspired multistable origami metamaterials for
  mechanical learning.
\newblock {\em Advanced Science}, 10(34):2305146, 2023.

\bibitem[Miu94]{Miura1994map}
Koryo Miura.
\newblock Map fold a la {Miura} style, its physical characteristics and
  application to the space science.
\newblock In R.~Takaki, editor, {\em Research of Pattern Formation}, pages
  77--90. KTK Scientific Publishers, Tokyo, Japan, 1994.

\bibitem[PLVP21]{Pratapa2021reprogrammable}
Phanisri~P. Pratapa, Ke~Liu, Siva~P. Vasudevan, and Glaucio~H. Paulino.
\newblock Reprogrammable kinematic branches in tessellated origami structures.
\newblock {\em Journal of Mechanisms and Robotics}, 13(3):031004, 2021.

\bibitem[Wat26]{Watts2026}
Simon Watts.
\newblock The degree-eight case of the {Miura}-ori flip-graph degree bound.
\newblock Zenodo, 2026.
\newblock \url{https://doi.org/10.5281/zenodo.22073084}.

\bibitem[Woo15]{Woods2015}
Kevin Woods.
\newblock Presburger arithmetic, rational generating functions, and
  quasi-polynomials.
\newblock {\em The Journal of Symbolic Logic}, 80(2):433--449, 2015.

\bibitem[ZWJ21]{Zhai2021metamaterials}
Zirui Zhai, Lingling Wu, and Hanqing Jiang.
\newblock Mechanical metamaterials based on origami and kirigami.
\newblock {\em Applied Physics Reviews}, 8(4):041319, 2021.

\end{thebibliography}

\end{document}